\documentclass[12pt]{amsart}      
\usepackage{txfonts}
\usepackage{amssymb}
\usepackage{eucal}
\usepackage{graphicx}
\usepackage{axodraw}
\usepackage{amsmath}
\usepackage{amscd}
\usepackage[all]{xy}           
\usepackage[active]{srcltx} 

\usepackage{amsfonts,latexsym}
\usepackage{xspace}
\usepackage{epsfig,exscale}
\usepackage{float}
\usepackage{color}
\usepackage{fancybox}
\usepackage{colordvi}
\usepackage{multicol}
\usepackage{colordvi}

\usepackage{pdfsync}

\usepackage{tikz}
\usepackage[colorlinks,final,backref=page,hyperindex,hypertex]{hyperref}

\newtheorem{theorem}{Theorem}[section]
\newtheorem{prop}[theorem]{Proposition}

\newtheorem{lem}[theorem]{Lemma}
\newtheorem{lemma}[theorem]{Lemma}
\newtheorem{coro}[theorem]{Corollary}
\newtheorem{prop-def}{Proposition-Definition}[section]

\theoremstyle{definition}
\newtheorem{defn}[theorem]{Definition}
\newtheorem{remark}[theorem]{Remark}

\newtheorem{conjecture}[theorem]{Conjecture}
\newtheorem{exam}[theorem]{Example}

\newcommand{\ignore}[1]{}

\newcommand{\nc}{\newcommand}
\newcommand{\delete}[1]{}

\nc{\mlabel}[1]{\label{#1}}  
\nc{\mcite}[1]{\cite{#1}}  
\nc{\mref}[1]{\ref{#1}}  
\nc{\mbibitem}[1]{\bibitem{#1}} 

\delete{
\nc{\mcite}[1]{\cite{#1}{{\bf{{\ }(#1)}}}}  
\nc{\mlabel}[1]{\label{#1}  
{\hfill \hspace{1cm}{\bf{{\ }\hfill(#1)}}}}
\nc{\mref}[1]{\ref{#1}{{\bf{{\ }(#1)}}}}  
\nc{\mbibitem}[1]{\bibitem[\bf #1]{#1}} 
}

\nc{\wvec}[2]{{\scriptsize{\big [ \!\!
    \begin{array}{c} #1 \\ #2 \end{array} \!\! \big ]}}}
\nc{\mrm}[1]{{\rm #1}}
\nc{\spair}[2]{\big[\begin{array}{c}\scs{#1} \\ \scs{#2} \end{array} \big]}

\nc{\rd}{\mrm{rd}}
\nc{\monic}{{\mathfrak M}}
\nc{\deff}{K}           
\nc{\coof}{L}           
\nc{\coef}{\bfk}        
\nc{\close}[1]{\overline{#1}}
\nc{\cl}{c}                 
\nc{\op}{o}                 
\nc{\io}{{io}}
\nc{\co}{\#}                 
\nc{\cc}{\mathcal{C}^c}      
\nc{\subd}{\mathrm{sub}}
\nc{\ccsp}{\calv^\cl}       
\nc{\ccrel}{\calw^\cl}      
\nc{\cccl}{\calh^\cl}       
\nc{\ocsp}{\calv^\op}       
\nc{\oc}{\mathcal{C}^o}      
\nc{\chenc}{\mathcal{CC}^o}  

\nc{\soc}{\mathcal{SC}^o}   
\nc{\doc}{\mathcal{DC}^o}    
\nc{\dsoc}{\mathcal{DSC}^o}  
\nc{\dosmc}{\mathcal{DTC}^o}    
\nc{\dfrsoc}{\mathcal{DFRSC}^o} 
\nc{\dfrcc}{\mathcal{DFRC}^c} 
\nc{\doch}{\mathcal{DCH}^o}    
\nc{\dcch}{\mathcal{DCH}^c}     
\nc{\dfrm}{\mathcal{DFRM}}   
\nc{\ocmzv}{\mathcal{CZV}^o} 
\nc{\oscmzv}{\mathcal{SZV}^o}  
\nc{\lzvset}{\mathcal{LZV}} 
\nc{\szvset}{\mathcal{SZV}} 
\nc{\spcc}{\mathcal{SpC}^c} 
\nc{\mcc}{\mathcal{MC}^c}    
\nc{\dcc}{\mathcal{DCC}^c}    
\nc{\cmzv}{\mathcal{DCZV}^c} 

\nc{\deco}[1]{\overline{#1}^o}  
\nc{\decc}[1]{\overline{#1}^c}  
\nc{\tn}{T}     
\nc{\dscc}{\mathcal{DSC}^c}  
\nc{\dsmc}{\mathcal{DMC}^c}  
\nc{\drc}{\mathcal{DRC}^c}  
\nc{\cdsmc}{\mathcal{CDMC}^c} 
\nc{\cs}{\mathcal{S}}   
\nc{\cm}{\mathcal{SM}}   
\nc{\ds}{\mathcal{DS}}  
\nc{\dm}{\mathcal{DM}}  

\nc{\smzv}{\mathcal{SZV}} 
\nc{\ocrel}{\calw^\op}      
\nc{\occl}{\calh^\op}       
\nc{\usp}{\{0\}}           
\nc{\cset}{\calc^\co}      
\nc{\csp}{\calv^\co}       
\nc{\crel}{\calw^\co}      
\nc{\ccl}{\calh^\co}       
\nc{\cone}[1]{\langle #1\rangle}
\nc{\ccone}[1]{\langle #1\rangle^\cl}
\nc{\ocone}[1]{\langle #1\rangle^\op}
\nc{\ocmzvset}{\mathcal{CZV}^o} 
\nc{\ccmzvset}{\mathcal{LZV}} 
\nc{\czvset}{\mathcal{CZV}}
\nc{\mzvset}{\mathcal{MZV}}
\nc{\dcset}{\calc^\co}      
\nc{\dcsp}{\calv^\co}       
\nc{\dcrel}{\calw^\co}      
\nc{\dccl}{\calh^\co}       
\nc{\dccset}{\calc^\cl}      
\nc{\dccsp}{\calv^\cl}       
\nc{\dccrel}{\calw^\cl}      
\nc{\dcccl}{\calh^\cl}       
\nc{\docrel}{\calw^\op}      
\nc{\doccl}{\calh^\op}       
\nc{\ccmzv}{\mathcal{CZV}^c}
\nc{\mzv}{\mathcal{MZV}}
\nc {\dtcp}{\mathcal {DTP}}

\nc{\dsf}{\mathcal{DSF}}    

\nc{\GL}{{\mrm GL}}
\nc{\SF}{{\mrm SF}} 
\nc{\depth}{{\mrm d}}
\nc{\ot}{\otimes}
\nc{\vep}{\varepsilon}
\nc{\id}{\mrm{id}}
\nc{\bfk}{{\bf k}}
\nc{\free}[1]{\bar{#1}}
\nc{\rk}{\mathrm{rk}}
\nc{\lc}{\lfloor}
\nc{\rc}{\rfloor}
\nc{\bt}{\bar{t}}
\nc{\dcup}{\atop{\cdot}{\cup}}
\nc{\lin}{\mathrm{lin}}
\nc{\Proj}{\mathrm{Proj}}
\nc{\dirlim}{\displaystyle{\lim_{\longrightarrow}}\,}
\nc{\im}{\mathrm{im}}

\nc{\bin}[2]{ (_{\stackrel{\scs{#1}}{\scs{#2}}})}  
\nc{\binc}[2]{ \left (\!\! \begin{array}{c} \scs{#1}\\
    \scs{#2} \end{array}\!\! \right )}  
\nc{\bincc}[2]{  \left ( {\scs{#1} \atop
    \vspace{-.5cm}\scs{#2}} \right )}  
\nc{\scs}[1]{\scriptstyle{#1}} 
\nc{\sha}{{\mbox{\cyr X}}}  
\nc{\ssha}{\mathop{\mbox{\scyr X}}} 
\font\cyr=wncyr10
\font\scyr=wncyr8

\nc{\BA}{{\mathbb A}} \nc{\CC}{{\mathbb C}} \nc{\DD}{{\mathbb D}}
\nc{\EE}{{\mathbb E}} \nc{\FF}{{\mathbb F}} \nc{\GG}{{\mathbb G}}
\nc{\HH}{{\mathbb H}} \nc{\LL}{{\mathbb L}} \nc{\NN}{{\mathbb N}}
\nc{\PP}{{\mathbb P}} \nc{\QQ}{{\mathbb Q}} \nc{\RR}{{\mathbb R}} \nc{\TT}{{\mathbb T}}
\nc{\VV}{{\mathbb V}} \nc{\ZZ}{{\mathbb Z}}

\nc{\cal}{\mathcal} \nc{\cala}{{\mathcal A}} \nc{\calc}{{\mathcal C}}
\nc{\calb}{{\mathcal B}} \nc{\cald}{{\mathcal D}}
\nc{\cale}{{\mathcal E}} \nc{\calf}{{\mathcal F}}
\nc{\calfr}{{{\mathcal F}^{\,r}}} \nc{\calfo}{{\mathcal F}^0}
\nc{\calfro}{{\mathcal F}^{\,r,0}} \nc{\oF}{{\overline{F}}}
\nc{\calg}{{\mathcal G}} \nc{\calh}{{\mathcal H}}
\nc{\cali}{{\mathcal I}} \nc{\calj}{{\mathcal J}}
\nc{\call}{{\mathcal L}} \nc{\calm}{{\mathcal M}}
\nc{\oM}{\overline{M}} \nc{\caln}{{\mathcal N}} \nc{\calo}{{\mathcal
O}} \nc{\calp}{{\mathcal P}} \nc{\calr}{{\mathcal R}}
\nc{\cals}{{\mathcal S}} \nc{\calt}{{\mathcal T}}
\nc{\caltr}{{\mathcal T}^{\,r}} \nc{\calu}{{\mathcal U}}
\nc{\calv}{{\mathcal V}} \nc{\calw}{{\mathcal W}}
\nc{\calx}{{\mathcal X}} \nc{\CA}{\mathcal{A}}

\nc{\fraka}{{\mathfrak a}} \nc{\frakB}{{\mathfrak B}}
\nc{\frakb}{{\mathfrak b}} \nc{\frakd}{{\mathfrak d}}
\nc{\oD}{\overline{D}}
\nc{\frakD}{{\mathcal D}} \nc{\frakf}{\mathfrak{f}}
\nc{\frakF}{{\mathfrak F}} \nc{\frakg}{{\mathfrak g}}
\nc{\frakI}{{\mathcal I}}
\nc{\frakL}{{\mathcal L}}
\nc{\frakm}{{\mathfrak m}}
\nc{\ofrakm}{\bar{\frakm}}
\nc{\frakM}{{\mathcal M}}
\nc{\frakMo}{{\mathfrak M}^0} \nc{\frakp}{{\mathfrak p}}
\nc{\frakP}{{\mathcal P}}
\nc{\frakR}{{\mathcal R}}
\nc{\frakS}{{\mathcal S}} \nc{\frakSo}{{\mathfrak S}^0}
\nc{\fraks}{{\mathfrak s}} \nc{\os}{\overline{\fraks}}
\nc{\frakT}{{\mathfrak T}}
\nc{\oT}{\overline{T}}
\nc{\frakV}{{\mathfrak V}}
\nc{\frakv}{{\mathfrak v}}
\nc{\oV}{\overline{V}}
\nc{\frakW}{{\mathcal W}}
\nc{\frakw}{{\mathfrak w}}
\nc{\oW}{\overline{W}}
\nc{\frakX}{{\mathfrak X}} \nc{\frakXo}{{\mathfrak X}^0}
\nc{\frakx}{{\mathbf x}}
\nc{\frakTx}{\frakT}      
\nc{\frakTa}{\frakT^a}        
\nc{\frakTxo}{\frakTx^0}   
\nc{\caltao}{\calt^{a,0}}   
\nc{\ox}{\overline{\frakx}} \nc{\fraky}{{\mathfrak y}}
\nc{\frakz}{{\mathfrak z}} \nc{\oX}{\overline{X}}

\nc{\zb}[1]{\textcolor{blue}{Bin: #1}}
\nc{\li}[1]{\textcolor{red}{Li: #1}}
\nc{\sy}[1]{\textcolor{purple}{Sylvie: #1}}

\begin{document}

\title{Conical zeta values and their double subdivision relations}

\author{Li Guo}
\address{Department of Mathematics and Computer Science,
         Rutgers University,
         Newark, NJ 07102, USA}
\email{liguo@rutgers.edu}

\author{Sylvie Paycha}
\address{Institute of Mathematics,
University of Potsdam,
Am Neuen Palais 10,
D-14469 Potsdam, Germany}
\email{paycha@math.uni-potsdam.de}

\author{Bin Zhang}
\address{Yangtze Center of Mathematics,
Sichuan University, Chengdu, 610064, P. R. China}
\email{zhangbin@scu.edu.cn}

\date{\today}

\maketitle

\begin{abstract}
We introduce the concept of a conical zeta value as a geometric generalization of a multiple zeta value in the context of convex cones. The quasi-shuffle and shuffle relations of multiple zeta values are generalized to open cone subdivision and closed cone subdivision relations respectively for conical zeta values. In order to achieve the closed cone subdivision relation, we also interpret linear relations among fractions as subdivisions of decorated closed cones. As a generalization of the double shuffle relation of multiple zeta values, we give the double subdivision relation of conical zeta values and formulate the extended double subdivision relation conjecture for conical zeta values.
\end{abstract}

\tableofcontents

\allowdisplaybreaks

\section{Introduction}
\mlabel{sec:int}

\subsection{Multiple zeta values and conical zeta values}

Multiple zeta values (MZVs) are special values of the multi-variable analytic function
\begin{equation}
\zeta(s_1,\cdots, s_k)=\sum_{n_1>\cdots>n_k>0} \frac{1}{n_1^{s_1}\cdots n_k^{s_k}}
\mlabel{eq:mzv}
\end{equation}
at integers $s_1\geq 2, s_i\geq 1, 1\leq i\leq k.$ Their
study in the two variable case goes back to Goldbach and Euler.
The general concept was introduced in the early 1990s, leading   to developments in     both mathematics~\mcite{Ho0,Za}, where MZVs conjecturally span (periods of) mixed Tate motives, and physics~\mcite{BK}, where MZVs  mysteriously appeared in Feynman integral computations. Since then
the subject has been studied intensively with interactions to a broad range of areas, including arithmetic geometry, combinatorics, number theory, knot theory, Hopf algebra, quantum field theory and mirror symmetry~\mcite{An,3BL,BB,Br,GM,GZ,Ho1,Ho3,IKZ,LM,Ra,Te,Za2}.

MZVs have several generalizations, such as Hurwitz multiple zeta values and multiple polylogarithms. In this paper we give a geometric generalization of MZVs in the context of convex cones. For an open convex cones   $C=\sum\limits_{i=1}^r \RR_{> 0} v_i$  spanned by vectors $v_i\in \ZZ^k_{\geq 0}, 1\leq i\leq r$, we define the conical zeta function associated with $C$ by
$$\zeta(C;s_1,\cdots,s_k):= \sum_{(n_1,\cdots,n_k)\in C\cap \ZZ^k_{\geq 0}} \frac{1}{n_1^{s_1}\cdots n_k^{s_k}}, s_i\in \CC, 1\leq i\leq k,
$$
where it converges and define a {\bf conical zeta value} to be the value of the function at nonnegative integer arguments.
Such values contain MZVs as special cases when the cones are taken to be Chen cones $\{x_1>\cdots >x_k>0\}$.

\subsection{Double shuffle and double subdivision relations}

A major goal in the study of MZVs is to determine all algebraic relations among the MZVs.
According to the {Double Shuffle Conjecture}~\mcite{Ho1,IKZ}, all such relations come from the shuffle and quasi-shuffle relations (the extended double shuffle relation) that encode products of MZVs from their summation and integration representations~\mcite{LM}. In this spirit, we generalize the double shuffle relation of MZVs to conical zeta values as subdivisions of open and closed cones.

The double shuffle relation can be summarized in the following commutative diagram that we will generalize to conical zeta values.
\begin{equation}
\xymatrix{(\calh^\ast_0,\ast) \ar[rd]_{\zeta^\ast} && \ \ (\calh^{\ssha}_0,\ssha)\  \ar@{>->>}[ll]^{\eta} \ar[ld]^{\zeta^{\ssha}} \\
& \QQ\mzvset &
}
\mlabel{eq:dsh}
\end{equation}
Here
\begin{enumerate}
\item
$\calh^\ast_0:= \QQ\, 1\oplus \hspace{-.5cm} \bigoplus\limits_{s_1,\cdots,s_k\geq 1, s_1\geq 2, k\geq 1} z_{s_1}\cdots z_{s_k}$
is the quasi-shuffle algebra~\mcite{H} with the quasi-shuffle product $\ast$, encoding the MZVs by the algebra homomorphism
$$\zeta^\ast: \calh^\ast_0\to \QQ\mzvset, \quad z_{s_1}\cdots z_{s_k}\mapsto \zeta(s_1,\cdots,s_k);
$$
\item
$\calh^{\ssha}_0:=\QQ\,1\oplus \hspace{-.5cm} \bigoplus\limits_{s_1,\cdots,s_k\geq 1,s_1\geq 2, k\geq 1} x_0^{s_1-1}x_1\cdots x_0^{s_k-1}x_1$
is the shuffle algebra with the shuffle product $\ssha$, encoding the MZVs by the algebra homomorphism
$$\zeta^{\ssha}: \calh^{\ssha}_0 \to \QQ\mzvset, \quad
x_0^{s_1-1}x_1\cdots x_0^{s_k-1}x_1$$
and
\item
$\eta: \calh^{\ssha}_0 \to \calh^\ast_0, x_0^{s_1-1}x_1\cdots x_0^{s_k-1}x_1\to z_{s_1}\cdots z_{s_k}, $
is the obvious linear bijection.
\end{enumerate}
The Double Shuffle Conjecture states that the kernel of $\zeta^\ast$ is the ideal of $\calh^\ast_1$ generated by the set
$$ \{ w_1\ast w_2 - \eta^{-1}(\eta(w_1)\ssha \eta(w_2)), \,|\, w_1\in \{z_1\}\cup \calh^\ast_0, w_2\in \calh^\ast_0\}.$$

The quasi-shuffle (stuffle) encoding $\zeta^\ast$ of MZVs follows directly from the definition of MZVs.
This is generalized to CZVs as open cone subdivisions.

The shuffle encoding $\zeta^{\ssha}$ is less direct. It is derived by the integral representation of MZVs~\mcite{LM} or, alternatively, from the integral representation of the multiple zeta fractions~\mcite{GX}. As shown there, the multiple zeta fractions
\begin{equation} \mathfrak{f}\wvec{s_1,\cdots,s_k}{u_1,\cdots,u_k}
:= \frac{1}{(u_1+\cdots+u_k)^{s_1}(u_2+\cdots+u_k)^{s_2}\cdots
u_k^{s_k}}, \quad s_i, u_i\geq 1, 1\leq i\leq k \mlabel{eq:pafr}
\end{equation}
on the one hand give multiple zeta values
\begin{equation}
\zeta(s_1,\cdots,s_k)=\sum_{u_1,\cdots,u_k\geq 1}
\mathfrak{f}\wvec{s_1,\cdots,s_k}{u_1,\cdots,u_k} \mlabel{eq:mzvpf}
\end{equation}
and on the other hand satisfy the shuffle relation. The shuffle relation of MZVs then follows by summing over the the $u_i$s.

To generalize the shuffle relation of MZVs to CZVs, we interpret the shuffle relation of multiple zeta fractions geometrically, starting from the simple observation that relations among fractions such as $$\frac{1}{u_1(u_1+u_2)}+\frac{1}{u_2(u_1+u_2)}= \frac{1}{u_1 u_2}$$ are related to subdivisions of cones, here the cone $x_1\ge 0, x_2\ge 0$ is seen as a union of the cones $x_1\ge x_2\ge 0$ and $x_2\ge x_1\ge 0$. More generally, we interpret linear relations among simple fractions i.e., of the type $\frac{1}{L_1L_2\cdots L_k}$ where $L_1,\cdots, L_k$ are linear independent linear forms, to subdivisions of simplicial cones. A differentiation procedure then further relates fractions $\frac{1}{L_1^{s_1}L_2^{s_2}\cdots L_k^{s_k}}$ with $s_i\in \NN$, to what we call algebraic subdivisions of decorated cones obtained from differentiating   geometric subdivisions of the underlying geometric cones. Such fractions arise as Laplace transforms: for example, if $L_i= u_i+\cdots +u_k$,  for $s_1>1,s_2\geq 1,\cdots, s_k\geq 1$, we observe that the multiple zeta fractions
$$ \mathfrak{f}\wvec{s_1,\cdots,s_k}{u_1,\cdots,u_k}= \frac{1}{L_1^{s_1}L_2^{s_2}\cdots
L_k^{s_k}} $$  can be written as differentiated Laplace transforms on the closed cone
$x_1\ge\cdots\ge x_k\ge 0$     (see Proposition (\mref {prop:DerFrac}) for notations):
\begin{eqnarray*}\mathfrak{f}\wvec{s_1,\cdots,s_k}{u_1,\cdots,u_k}&=& \frac
1{(s_1-1)!\cdots (s_k -1)!}\partial_{L_1^*}^{s_1-1}\cdots
\partial_{L_k^*}^{s_k-1}	
\frac{1}{L_1\, \cdots
L_k }\\
&=& \frac
1{(s_1-1)!\cdots (s_k -1)!}\partial_{L_1^*}^{s_1-1}\cdots
\partial_{L_k^*}^{s_k-1}\int_{x_1>\cdots>x_k>0} e^{-\sum_{i=1}^k x_i u_i}\, dx_1\cdots dx_k.
\end{eqnarray*}
This is the starting point for our generalization of the double shuffle relation among zeta values associated  with more general cones.
Eventually we obtain a geometric interpretation of the commutative diagram in Eq.~(\mref{eq:dsh}) and generalize it to a commutative diagram in Eq.~(\mref{eq:dsub}) of double subdivision relation for CZVs.

\subsection{Layout of the paper}

After summarizing concepts and basic facts on convex cones, we give in Section~\mref{sec:ocone} the definition of conical zeta values and their open subdivision relation as a generalization of the stuffle (quasi-shuffle) relation of MZVs. In order to generalize the shuffle relation of MZVs to CZVs, we generalize the shuffle relation of multiple zeta fractions to a suitable relation for a much larger class of fractions derived from CZVs by means of a differentiation procedure similar to the one described above in the case of MZVs.
We achieve this in two steps. In   Section~\mref{sec:ccone}, we relate via a bijection closed cones modulo subdivisions to simple fractions. Thus linear relations among simple fractions are precisely those coming from subdivisions of closed simplicial cones. By means of the natural differential structure on fractions, in Section~\mref{sec:dcone} we infer from this bijection a one to one correspondence between decorated cones modulo subdivision and pure fractions. This correspondence between cones and fractions is applied in Section~\mref{sec:czv} to provide  closed subdivision relations of CZVs when  expressed as Shintani zeta values. In doing so, the shuffle product of multiple zeta fractions seen as decompositions of fractions with linear poles is reflected geometrically as subdivisions of the closed Chen cones. Combining the open and closed subdivision relations with the concept of cone pairs gives the double subdivision relation of CVZs that generalizes the double shuffle relation of MZVs. Finally it is shown that CVZs and Shintani zeta values span the same linear space.
\smallskip

In this paper we shall not touch on divergent conical zeta values, which will be the subject of a forthcoming paper. Divergent MZVs, which in recent years have been studied in the algebraic framework of Connes and Kreimer~\mcite{CK} inspired by the method of renormalization of quantum field theory, can be defined using several approaches such as ~\mcite{BV,GZ,MP}. In this forthcoming paper, we construct  a coalgebra structure on cones and as  an application, we show that renormalization of conical zeta values recovers the local Euler-Maclaurin formula~\mcite{BV,GM}.

\section {Convex cones and conical zeta values}
\mlabel{sec:ocone}

\subsection {Polyhedral cones}
We first collect basic notations and facts (mostly following~\mcite{Fu} and \mcite{Zi})  on cones that will be used in this paper.
Let $\deff\subseteq \RR$ be a field and let $k\geq 0$ be an integer. In practice $\deff$ is usually the field $\QQ$ of rational numbers.
\begin{enumerate}
\item
A {\bf closed (polyhedral) cone} (resp. An {\bf open (polyhedral) cone}) in $\deff^k$ is the convex set
\begin{eqnarray}
&&\ccone{v_1,\cdots,v_n}:=\deff_{\geq 0}v_1+\cdots+\deff_{\geq 0}v_n, \mlabel{eq:opencone} \\
\text{ (resp. } &&\ocone{v_1,\cdots,v_n}:=\deff_{> 0}v_1+\cdots+\deff_{> 0}v_n \text{)},
 \quad v_i\in \deff^k_{\geq 0}, 1\leq i\leq n.
\notag
\end{eqnarray}
\item A cone is always taken to be a closed or open polyhedral cone.
In particular, the term polyhedral will be omitted in this paper and we sometimes write $ \langle v_1,\cdots,v_n\rangle$ when the closedness or the openness does not play any role.
\item
For a cone $C=\cone{v_1,\cdots,v_n}$ in $\deff^k$ and a field $\coof \subseteq \RR$, let $C(\coof)$ denote $\coof_{\geq 0} v_1+\cdots +\coof_{\geq 0}v_n$ if $C$ is closed and $\coof_{> 0} v_1+\cdots +\coof_{> 0}v_n$ if $C$ is open.
\item
The set $\{v_1,\cdots,v_n\}$ in the definition of a cone is called the {\bf generating set} or the {\bf spanning set} of the cone. The dimension of the $\deff$-linear subspace generated by the cone is called its {\bf dimension}.
\item A closed cone in $\deff^k$ can also be described as the intersection $\cap_i H_{u_i}$ of finitely many half spaces $H_{u_i}=\{x\in \deff ^k\  | u_i(x)\ge 0\}$ defined by linear functionals $u_i$ with $\deff$-coefficients on $\deff^k$ (see e.g. Theorem 1.3 in \mcite{Zi}).
\item
Let  $\cc_k(\deff)$ (resp. $\oc_k(\deff)$) denote the set of  closed (resp. open cones) in $\deff^k$, $k\geq 1$. For $k=0$ we set
  $\cc_{0}(\deff)=\{0\}$ (resp. $\oc_{0}(\deff)=\{0\}$) by convention.
The natural inclusions $\cc_k(\deff) \to \cc_{k+1}(\deff)$ (resp. $\oc_k(\deff)\to \oc_{k+1}(\deff)$) induced by the natural inclusion $\deff^k \to \deff^{k+1}$,
give rise to the direct  limit set $\cc(\deff)=\dirlim \cc_k(\deff)$ (resp. $\oc(\deff)=\dirlim \oc_k(\deff)$).
\item
A {\bf simplicial cone} or a {\bf simplex cone} is a cone spanned by linearly independent vectors.
\item
A cone in $\QQ^k$ is called a {\bf rational cone}. Thus a rational cone is spanned by vectors in $\QQ^k$ (equivalently in $\ZZ^k$).
\item
A {\bf smooth cone} is a rational cone with a spanning set that is a part of a basis of $\ZZ^k\subseteq \RR^k$. In this case, the spanning set is unique and is called the  {\bf  primary set} of the cone.
\item
A cone is called {\bf strongly convex} if it does not contain any linear subspace.
\item A {\bf face} of a closed cone $\ccone{v_1,\cdots,v_n}$ in $\deff^k$ is a subset of the form $\ccone{v_1,\cdots,v_n}\cap \{u=0\}$, where $u:\deff^k\to \deff$ is a linear function with $\deff$-coefficients which is non-negative on $\ccone{v_1,\cdots,v_n}$.
\item  A {\bf face} of an open cone $\ocone{v_1,\cdots,v_n}$ in $\deff^k$ is an open cone of the form $\ocone{v_{i_1},\cdots,v_{i_r}}(\deff)$ where $\ccone{v_{i_1},\cdots,v_{i_r}}(\deff)$ is a face of $\ccone{v_1,\cdots,v_n}$.
\item A face  $F$ of a cone $C$   is again a cone and we write $F\leq C$.
If $F$ is a proper face  of a cone $C$ we write $F<C$.
 A 1-dimensional face is called an {\bf edge}. A codimension 1 face is called a {\bf facet}.
\item
For $\vec{x}=(x_1,\cdots ,x_k)$ and $\vec{y}=(y_1,\cdots ,y_k)$ in $\RR^k$, let $(\vec{x},\vec{y})$ denote the inner product $x_1y_1+\cdots +x_ky_k$. Through this inner product, $\RR^k$ is identified with its own dual space $(\RR^k)^*$.
\end{enumerate}

\subsection {Subdivision of cones}

In this subsection, we recall some facts about subdivisions of polyhedral cones. For the sake of completeness, we provide proofs for some of the results.

\begin {defn}
{\rm
\begin {enumerate}
\item
A {\bf subdivision of a closed cone $C\in \cc_k(\deff)$} is a set $\{C_1,\cdots,C_r\}\subseteq \cc_k(\deff)$ such that
\begin{enumerate}
\item[(i)] $C=\cup_{i=1}^r C_i$,
\item[(ii)] $C_1,\cdots,C_r$ have the same dimension as $C$ and
\item[(iii)] intersect along their faces i.e.,  $
C_i\cap C_j$ is a face of both $C_i$ and $C_j$.
\end{enumerate}
\item
A {\bf subdivision of an open cone $C$} is the set of the relative interiors of the closed cones in
\begin{equation}
\{\cap_{j=1}^t D_{i_j}\ |\ \{i_1,\cdots,i_t\}\subseteq \{1,\cdots,r\}, 1\leq t\leq r\}
\mlabel{eq:odiv}
\end{equation}
where $\{D_1,\cdots,D_t\}$ is a subdivision of the closure $\close{C}$ of $C$. By convention, the relative interior of $\{0\}$ is $\{0\}$.
\item
A subdivision of a rational cone is called {\bf smooth} if all the cones in the subdivision are smooth.
\end{enumerate}
}
\end{defn}
\begin {prop}
\begin{enumerate}
\item
Any cone in $\deff^n$ can be subdivided into strongly
convex simplicial cones in $\deff^n$.
\mlabel{it:ConeToSim}
\item
Any strongly convex simplicial rational cone can be subdivided into smooth cones.
\mlabel{it:ratsm}
\end{enumerate}
\mlabel{pp:ConeToSim}
\mlabel{pp:ratsm}
\end{prop}
\begin{proof}
(\ref{it:ConeToSim})
By taking the intersections with coordinate orthants, we can assume that the cone is strongly convex.

Now for a strongly convex cone $C$ in $\deff^n$, we take its barycenter type subdivision built as follows. In the following  we identify a point $M$ in $\deff^n$ with the vector $\vec{OM}=v$. For each face $F$ of $C$, take a vector $v_F\in \deff^n$ in the relative interior of $F$. Note that, since $F(\QQ)\supseteq F$ and $F(\QQ)$ is dense in $F(\RR)$, such a vector always exists. Let $n=\dim(C)$. If the cone $C$ is open, then the open cones $\ocone{v_{F_1}, \cdots v_{F_\ell}}\subset C$ with $F_1 <\cdots <F_\ell, 0\leq \ell\leq n,$ are simplicial and intersect along their faces. If the cone $C$ is closed, then the closed cones $\ccone{v_{F_1}, \cdots, v_{F_n}}$ with $F_1 <\cdots <F_n,$ are simplicial and intersect along their faces. Thus to prove that this gives a subdivision, we only need to prove that the union of these cones is $C$.

First we prove this for a closed cone $C$. We proceed by induction on the dimension $n$ of $C$. Since the case   $n=1$ is trivial, we assume that $n\geq 2$. For any vector $v$ in $C$, if $v$ is a multiple of the vector $v_C$ chosen in the relative interior of $C$ as above, then we have the conclusion. Otherwise the vectors $v$ and $v_C$ span a 2-dimensional linear space $V$. Let $\deff_{\geq 0}v_1$ and $\deff_{\geq 0}v_2$ be the outmost intersections. Then by convexity, $\langle v_1, v_2\rangle$ is in $C(\deff)$. By the choice of $v_1$ and $v_2$, $\langle u_1,u_2\rangle = V\cap C$. Thus the intersection of $V$ with the boundary of $C$ are the two rays $\deff_{\geq 0}u_1$ and $\deff_{\geq 0}u_2$. The vectors $v_1$, $v_2$ lie  in some facets since the boundary of $C$ is the union of its facets.  Then $v$ is a non-negative $\deff$-linear combination of $v_C$ with one of the two vectors  $v_1$ or $v_2$, say $v_1$.
By the induction hypothesis, $v_1$ is in $\ccone{v_{F_1},\cdots, v_{F_{n-1}}}$ with $F_{v_1}<\cdots<F_{v_{n-1}}$ where $F_{v_{n-1}}$ is a facet of $C$. Then $v$ is in one of the simplicial subdivisions
$\ccone{v_{F_1},\cdots, v_{F_{n-1}},v_C}(\deff).$

For an open cone, the proof is similar. The only difference is that the intersection rays $\deff_{> 0}v_1$ and $\deff_{> 0}v_2$ may be in the interior of some lower dimensional faces of $C$. Let $v$ be a positive $\deff$-linear combination of $v_C$ and $v_1$ as in the closed cone case. Then by the induction hypothesis, $v_1$ is in $\ocone{v_{F_1},\cdots, v_{F_r}}$ with $F_{v_1}<\cdots<F_{v_r}$ and $\dim F_{v_r}\leq n-1$. Thus $v$ is in the simplicial subdivisions $\ocone{v_{F_1},\cdots, v_{F_r},v_C}.$
\smallskip

\noindent
(\ref{it:ratsm}) See the second exercise on page 48  of \cite {Fu}.
\end{proof}

\begin{lem}
\begin{enumerate}
\item For a family of closed cones $\{C_{i}\}$ in $\deff^k$, $1\leq i\leq m$,  that span the same linear subspace of $\deff^n$, there is a simplicial subdivision $\{C_{ij}\}$ in $\deff^k$ of $C_i$ such that any two of $C_{ij}$ either coincide or only intersect along their faces.
\mlabel{it:conesub}
\item For a family of rational closed cones $\{C_i\}$, $1\leq i\leq m$, that span the same linear subspace of $\QQ^n$,
there is a smooth subdivision $\{C_{ij}\}$ of $C_i$ such that every two of $C_{ij}$ either coincide or only intersect along their faces.
\mlabel{it:smoothsub}
\end{enumerate}
\mlabel{lem:ConeSubd}
\end{lem}

\begin{proof}
(\ref{it:conesub}) Each closed cone $C_i$ can be written as an intersection of half hyperplanes $ H_{u^i_j}, j=1,\cdots, s_j$
$$C_i=\bigcap_{j=1}^{s_i} H_{u^i_j}, 1\leq i\leq m.$$

Denote
$$S:=\{(i,j)\ |  1\leq i\leq m, \ 1\leq j\leq s_i\},
$$
and
$$F:=\{\nu: S\to \{1,-1\} \ |\ \exists i_0, \nu(i_0,j)=1, \forall j \}.$$
Each element  $\nu\in F$ defines a set
$$C_\nu:=\bigcap_{i=1}^m \bigcap _{j=1}^{s_i} H_{\nu ((i,j))u^i_j},
$$
which is a cone though it may be trivial.

For each $1\leq i_0\leq m$, consider the set
$$F_{i_0}:=\{\nu\in F\,|\, \nu(i_0,j)=1\text{ for all } j\text{ and } \dim C_\nu=n\}.$$
Then
$ C_{i_0}=\bigcup _{\nu \in F_{i_0}}C_\nu.$
Further for $\nu, \mu \in \cup_{i=1}^m F_i$, we have
$$ C_\mu \cap C_\nu = \bigcap_{(i,j)\in S} H_{\mu(i,j)u^i_j}\cap H_{\nu(i,j)u^i_j}
=\left(\bigcap_{\mu(i,j)=\nu(i,j)} H_{\mu(i,j)u^i_j}\right) \bigcap\left( \bigcap_{\mu(i,j)\neq \nu(i,j)}\left\{u^i_j=0\right\}\right).
$$
We note that the faces of a closed cone $C:=\cap_{\ell=1}^k H_{u_\ell}$ are of the form
$\left(\bigcap_{\ell\in K'} H_{u_\ell} \right)\cap \left( \bigcap_{\ell\in K''} \left\{u_\ell=0\right\} \right)$ for a partition $K'\sqcup K''=[k]$. Thus $C_\mu\cap C_\nu$ is a face of both $C_\nu$ and $C_\mu$.

Therefore, for $1\leq i\leq m$, the set $\{C_\nu\,|\, \nu\in F_i\}$ is a subdivision of $C_i$ and any two $\mu, \nu\in \cap_{i=1}^m F_i$, $C_\mu$ and $C_\nu$ either coincide or only intersect along their faces.
By subdividing each $C_\nu, \nu\in \cup_{i=1}^m F_i$ into simplicial cones applying Proposition \mref {pp:ConeToSim}.(\mref{it:ConeToSim}), we obtain the simplicial subdivisions of $C_i$ in the proposition.

\smallskip

\noindent
(\mref{it:smoothsub}) Taking $\deff=\QQ$ in Item~(\mref{it:conesub}), and applying Proposition~\mref{pp:ConeToSim}.(\ref{it:ratsm}) to further subdivide each $C_{ij}$ there into smooth cones, we achieve the desired subdivision.
\end{proof}

\subsection{Conical zeta values and open subdivision relations}
\mlabel{ss:osub}
We now introduce our main concept of study in this paper.
\begin{defn}
{\rm
Let $C$ be an open cone in $\RR ^k_{\geq 0}$ and let $\vec
s \in \CC ^k$. Define the {\bf conical zeta function} by
\begin{equation}
\zeta(C;\vec s)=\sum_{(n_1,\cdots,n_k)\in C\cap \ZZ^k}
\frac{1}{n_1^{s_1}\cdots n_k^{s_k}},
\mlabel{eq:ocmzv}
\end{equation}
if the sum converges. Here we have used the convention that $0^s=1$ for any $s$.
}
\mlabel{de:cmzv}
\end{defn}

Note that with this convention, the definition of $\zeta(C;\vec{s})$ does not depend on the integer $k$ such that $C\subseteq \RR^k_{\geq 0}$ and $\vec{s}\in \CC^k$. Thus we can use $\zeta(C;\vec{s})$ without referring to $k$. When
$s_1,\cdots,s_k$ are taken to be integers, we call
$\zeta(C;\vec s)$ a {\bf conical zeta value (CZV)}. Sometimes we use the name {\bf open conical zeta value} also.
Let $\ocmzvset$ denote the set of convergent conical zeta values. Define the space $\QQ\ocmzvset$ to be the {\bf space of convergent conical zeta values over $\QQ $.}

\begin {lem} Let $C$ be an open cone in $\RR^k_{\geq 0}$. For $\vec s \in \ZZ ^n$ with $s_i\ge 2$, $\zeta(C;\vec s)$ is convergent.
\end{lem}

\begin{proof} Let $C$ be the first coordinate orthant $\RR^n_{\geq 0}$. Then
$$\zeta (C;\vec s)= \prod_{i=1}^k \zeta(s_i)$$
and hence is convergent if $s_i\geq 2$ for $1\leq i\leq n$.
Then the statement holds for any open cone $C\subseteq \RR^n_{\geq 0}$ since
$\zeta(C;\vec{s})\leq \zeta(\ZZ^n_{\geq 0};\vec{s})$.
\end{proof}

An (open or closed) {\bf Chen cone} of dimension $k$ is a (open or closed) cone $C_{k,\sigma}$ spanned
by the vectors $\{e_{\sigma(1)},e_{\sigma(1)}+e_{\sigma(2)},\cdots,
e_{\sigma(1)}+\cdots + e_{\sigma(k)}\}$ where $\{e_1,\cdots,e_n\}$
is the standard basis of $\ZZ^n$ and $\sigma\in S_n$, $S_n$ is the symmetric group on $\{1,\cdots,n\}$. Let $C_k$
denote the standard (open or closed) Chen cone spanned by $\{e_1,e_1+e_2,\cdots,e_1+\cdots + e_k\}$.

\begin{prop}
For any open Chen cone $C_{k,\sigma}$, $k\geq 1, \sigma\in S_n$, we have
$$ \zeta({C_{k,\sigma}};s_1,\cdots,s_n)=\zeta(s_{\sigma(1)}, \cdots,s_{\sigma(k)}),$$
where the right hand side is the multiple zeta value.
\mlabel{pp:mzv}
\end{prop}

Therefore the space $\QQ \mzvset $ spanned by MZVs over $\QQ$ is a subspace of $\QQ \ocmzvset$.

\begin{proof}
An element of $C_{k,\sigma}\cap \ZZ^n$ is of the form
\begin{align*} &a_1e_{\sigma(1)}+a_2(e_{\sigma(1)}+e_{\sigma(2)}) +\cdots + a_k(e_{\sigma(1)}+\cdots + e_{\sigma(k)}) \\ =&
(a_1+\cdots+a_k)e_{\sigma(1)} +(a_2+\cdots+a_k)e_{\sigma(2)}+\cdots +a_ke_{\sigma(k)},
\end{align*}
where $a_i\in (0,\infty), 1\leq i\leq k$. Hence
$$\zeta(C_{k,\sigma};s_1,\cdots,s_n) =\sum_{a_1,\cdots,a_k\geq 1} \frac{1}{(a_1+\cdots+a_k)^{s_{\sigma(1)}}\cdots a_k^{s_{\sigma(k)}}} =\zeta(s_{\sigma(1)},\cdots,s_{\sigma(k)}).$$
\end{proof}

From the definition of CZVs, we derive the following lemma.
\begin{lemma}
Let  $\{C_i\}_i$ be a family of open cones that form a subdivision of an open cone $C$, then
\begin{equation}
\zeta(C;\vec{s})=\sum_i \zeta(C_i;\vec{s}).
\mlabel{eq:subd}
\end{equation}
\mlabel{lem:subd}
This is called an {\bf open subdivision relation} of CZVs.
\end{lemma}

We next show that open subdivision relations of open Chen cones recover the quasi-shuffle relations of MZVs.
First recall the quasi-shuffle encoding of MZVs. Define \begin{equation}
\calh^*_0:= \QQ\,1 \oplus\hspace{-.5cm} \bigoplus_{s_1,\cdots,s_k\geq 1, s_1\geq 2, k\geq 1} z_{s_1}\cdots z_{s_k},
\mlabel{eq:qsh}
\end{equation}
with the quasi-shuffle product $\ast$ which is defined recursively but can also be defined by the stuffle product as follows. Define
$$ St_{k,\ell}=\left \{(\varphi,\psi)\,\Big |\, \begin{array}{l}\varphi:[k]\to [m], \psi:[\ell]\to [m] \text{ are order preserving,}\\
\text{ injective and } \im(\varphi)\cup \im(\psi)=[m] \end{array}\right \}$$
and $\varphi^{-1}(i)=0$ when $\varphi^{-1}(i)=\emptyset$.
Then
\begin{equation}
z_{s_1}\cdots z_{s_k} \ast z_{s_{k+1}}\cdots z_{s_{k+\ell}} = \sum_{(\varphi,\psi)\in St_{k,\ell}} z_{s_{\varphi^{-1}(1)}+s_{k+\psi^{-1}(1)}}\cdots z_{s_{\varphi^{-1}(m)}+s_{k+\psi^{-1}(m)}}.
\mlabel{eq:stu}
\end{equation}
Then the quasi-shuffle encoding of MZVs is given by the algebra homomorphism
$$ \zeta^\ast: \calh^*_0 \to \QQ \mzv, \quad z_{s_1}\cdots z_{s_k} \mapsto \zeta(s_1,\cdots,s_k),$$
namely,
$$\zeta(s_1,\cdots,s_k)\zeta(s_{k+1},\cdots,s_{k+\ell}) =\zeta^\ast (z_{s_1}\cdots z_{s_k} \ast z_{s_{k+1}}\cdots z_{s_{k+\ell}}).$$

We likewise give an open cone encoding of CZVs.
\begin{defn}
{\rm
\begin{enumerate}
\item
Let $\doc$ be the set of {\bf decorated open cones} consisting of pairs $(C;\vec{s})$  where $C\subseteq \QQ^k_{\geq 0}$ is an open rational cone and $\vec{s}\in \ZZ^k_{\geq 0}$.
\item
Let $\{C_i\}_i$ be an open subdivision of $C$. Then we also call $\{(C_i;\vec s)\}_i$ an open subdivision of $(C;\vec s)$ and denote it by
$$ (C;\vec{s}) \prec \sum_i (C_i;\vec s).$$
\item
Denote $\doc_0$ for the subset of $\doc$ such that $\zeta(C;\vec{s})$ is convergent.
\item
Define the linear map
\begin{equation}
\zeta^o: \QQ\doc_0 \to \QQ\ocmzvset, \quad (C;\vec{s})\mapsto \zeta(C;\vec{s}).
\mlabel{eq:zetao}
\end{equation}
\item
Let $\doch$ (resp. $\doch_0$) denote the subset of $\doc$ of decorated cones  $(C, \vec s)$ with underlying cone $C$ an open Chen cone (resp. that give convergent CZVs).
\end{enumerate}
}
\mlabel{de:doc}
\end{defn}
By definition, $\zeta^o(C;\vec{s})=\zeta(C;\vec{s})$. Thus the two notations will be used interchangably; the notation $\zeta^o(C;\vec{s})$ will be used to stress the map $\zeta^o$.

Then we have the bijection
$$ \calh^\ast_0\to \QQ\doch_0, z_{s_1}\cdots z_{s_k} \mapsto (\ocone{e_1,\cdots, e_1+\cdots+e_k};s_1,\cdots,s_k),$$
completing the following commutative diagram of linear maps.

\begin{equation}
\xymatrix{\calh^\ast_0 \ar@{>->>}[r] \ar[d]^{\zeta^\ast}& \QQ\doch_0 \ar@{^{(}->}[r] & \QQ\doc_o \ar[d]^{\zeta^o} \\
\QQ\mzvset \ar@{^{(}->}[rr] && \QQ\ocmzv
}
\mlabel{eq:osub}
\end{equation}
In this context, we obtain the following
\begin{theorem}
\begin{enumerate}
\item
The quasi-shuffle product in Eq.~$($\mref{eq:stu}$)$ corresponds to the open subdivision
\begin{eqnarray}
&&\ocone{e_1,\cdots,e_1+\cdots+e_k, e_{k+1},\cdots,e_{k+1}+\cdots+e_{k+\ell}} \mlabel{eq:chsubd}\\ &=&\coprod_{(\varphi,\psi)\in St_{k,\ell}} \ocone{e_{\varphi^{-1}(1)}+e_{k+\psi^{-1}(1)}, \cdots, e_{\varphi^{-1}(1)}+e_{k+\psi^{-1}(1)}+\cdots +e_{\varphi^{-1}(m)}+e_{k+\psi^{-1}(m)}}.
\notag
\end{eqnarray}
in the sense that the composition in the top of Eq.~$($\mref{eq:osub}$)$ sends the right hand side of Eq.~$($\mref{eq:stu}$)$ to the right hand side of Eq.~$($\mref{eq:chsubd}$)$.
\item
The quasi-shuffle product of MZVs:
$$\zeta(s_1,\cdots,s_k)\zeta(s_{k+1},\cdots,s_{k+\ell}) =\zeta^*(z_{s_1}\cdots z_{s_k} \ast z_{s_{k+1}}\cdots z_{s_{k+\ell}})$$
coincides with the subdivision of CZVs:
\begin{eqnarray*}
&&\zeta^o(\ocone{e_1,\cdots,e_1+\cdots+e_k, e_{k+1},\cdots,e_{k+1}+\cdots+e_{k+\ell}}
;s_1,\cdots,s_k, s_{k+1},\cdots,s_{k+\ell})\\ &=&\sum_{(\varphi,\psi)\in St_{k,\ell}}
\zeta^o(\ocone{e_{\varphi^{-1}(1)}+e_{k+\psi^{-1}(1)}, \cdots, e_{\varphi^{-1}(1)}+e_{k+\psi^{-1}(1)}+\cdots +e_{\varphi^{-1}(m)}+e_{k+\psi^{-1}(k)}};s_1,\cdots,s_m).
\end{eqnarray*}
\end{enumerate}
\mlabel{thm:stsub}
\end{theorem}

For example, the closed cone subdivision
$$ \ccone{e_1,e_1}=\ccone{e_1,e_1+e_2}\cup \ccone{e_2,e_1+e_2}$$
by Chen cones gives the open cone subdivision
$$(\ocone{e_1,e_2};s_1,s_2)=(\ocone{e_1, e_1+e_2};s_1,s_2)\sqcup(\ocone{e_2, e_1+e_2};s_1,s_2)\sqcup (\ocone{e_1+e_2};s_1,s_2)$$
which recovers the quasi-shuffle relation
$$ z_{s_1}\ast z_{s_2}=z_{s_1}z_{s_2}+z_{s_2}z_{s_1}+z_{s_1+s_2}.$$
Indeed we have
\begin{eqnarray*}
\zeta(s_1)\zeta(s_2)&=&\zeta^o(\ocone{e_1,e_2},s_1,s_2)\\ &=&\zeta^o(\ocone{e_1, e_1+e_2},s_1,s_2)+\zeta^o(\ocone{e_2, e_1+e_2},s_1,s_2)+\zeta^o(\ocone{e_1+e_2},s_1,s_2)\\
&=&\zeta (s_1,s_2)+\zeta(s_2,s_1)+\zeta(s_1+s_2).
\end{eqnarray*}

\section {Closed cones and simple fractions}
\mlabel{sec:ccone}
As noted in the introduction, we will establish a class of relations of CZVs that generalizes the shuffle relation of MZVs. Motivated by the approach of multiple zeta fractions outlined in the introduction, we first relate CZVs to a class of fractions and generalize the shuffle relation of multiple zeta fractions to this class of fractions. Our geometric approach of generalizing the shuffle relation consists in encoding all linear relations of these CZV fractions as subdivision relations of closed cones from the CZVs. Thus we now make a digression of our discussion of CZVs to relate closed cones with a family of rational functions, which we call simple fractions. Under this correspondence we show that linear relations among  simple  fractions have a natural geometric interpretation as subdivisions of closed cones.
This correspondence will be generalized to decorated cones in the next section.

\subsection {From closed cones to simple fractions}
\mlabel{subsec:sfrac}

\begin {defn} {\rm
Let $\deff$ be a subfield of $\RR$. Let $z_i, i\geq 1$ be a countable set of variables and let $\vec{z}=(z_i)_{i\geq 1}$. A {\bf simple fraction} with coefficients in $\deff$ is a fraction of the form $ \frac {1}{L_1\cdots L_k}$, where $L_1, \cdots, L_k \in \deff [\vec{z}]$ are linearly independent linear functions.
Let $\cals (\deff)$ be the $\deff$-linear subspace of the quotient field $\deff(\vec{z})$ of $\deff [\vec{z}]$ generated by simple fractions with coefficients in  $\deff$.
}
\end{defn}

Let $\deff\cc(\deff)$ denote the $\deff$-vector space spanned by $\cc(\deff)$.
We will define a natural map from $\deff\cc(\deff)$ to $\cals(\deff)$. Lawrence in \mcite {La} (see also \mcite {GP}) constructed a similar map based on the valuation property.
Our map generalizes Lawrence's map in so far as it takes non-zero values for lower dimensional cones in large dimensional spaces.

Let $C$ be a closed simplicial cone in $\deff^n_{\geq 0}$ with linearly independent generators $v_1, \cdots v_k$. Let $e_i, 1\leq i\leq n$, be the standard basis of $\deff^n$. For $1\leq i\leq k$, let $v_i=\sum\limits_{j=1}^n a_{ji}e_j, a_{ji}\in \deff$. Define linear functions
$L_i=L_{v_i}=\sum\limits_{j=1}^n a_{ji}z_j$ and let $A_C=[a_{ij}]$ denote the associated matrix in $M_{n\times k}(\deff)$ with $v_i$ as column vectors.
Let $w(v_1, \cdots,
v_k)$ or $w(C)$ denote the sum of absolute values of the determinants of all minors of $A_C$ of rank $k$. Then define
\begin{equation}
\Phi_n(C):=\frac
{w(v_1, \cdots , v_k)}{L_1\cdots L_k}.
\mlabel{eq:phi0}
\end{equation}
This defines a map
$\Phi_n$ from the set $\cc(\deff)$ of closed simplicial cones in $\deff^n_{\geq 0}$ to $\cals (\deff)$.

\begin{lem}
\begin{enumerate}
\item
Let $C=\ccone{v_1,\cdots,v_n}$ be a closed cone of rank $n$ in $\deff^n$, then
$$\Phi_n (C)=(-1)^n\int \cdots \int _{C(\RR)} \exp(x_1z_1+\cdots +x_nz_n)dx_1\cdots dx_n,
$$
where $C(\RR)$ is the $\RR_{\geq 0}$-linear span of $C$ and $\vec z$ is any element in $$\check{C}^-:=\{\vec y\ |\ (\vec y,c)<0, \forall c\in C(\RR)\}.$$
\mlabel{it:phiint}
\item
Let $C$ be a
closed simplicial cone in $\deff^k$ and let $\{C_1,\cdots, C_r \}$ be a closed subdivision of $C$ into closed simplicial cones $C_1,\cdots,C_r$ in $\deff^k$. Then $\Phi_n (C)=\sum\limits_{i=1}^r \Phi_n (C_i)$.
\mlabel{it:phidiv}
\end{enumerate}
\mlabel{lem:phi}
\end {lem}

\begin{proof}
(\mref{it:phiint}) Since $C$ is strongly convex, $\check{C}^-$ is a cone of rank $n$.
With our notation, we have
$$ (v_1,\cdots,v_n)=(e_1,\cdots,e_n)A,\quad (L_1,\cdots,L_n)=(z_1,\cdots,z_n)A.
$$
Any point $\sum x_ie_i$ in $C$ can be uniquely expressed as $\sum y_iv_i$ through a change of variables
$$(x_1,\cdots,x_n)^T = A (y_1,\cdots,y_n)^T.$$
Thus we obtain
\begin{eqnarray*}
&&
\int\cdots \int_{C(\RR)} \exp(x_1z_1+\cdots+x_nz_n) d x_1\cdots d x_n\\
&=& \int_0^{+\infty} \cdots \int_0^{+\infty} \exp\big( (z_1,\cdots,z_1) A (y_1,\cdots,y_n)^T\big) |\det(A)|\, dy_1\cdots dy_n\\
&=& \int_0^{+\infty} \cdots \int_0^{+\infty}
\exp\big((L_1,\cdots,L_n) (y_1,\cdots,y_n)^T\big) |\det(A)|\, dy_1\cdots dy_n\\
&=& |\det(A)| \prod_{i=1}^n \int_0^\infty e^{L_i y_i} dy_i\\
&=& (-1)^n\frac{|\det(A)|}{L_1\cdots L_n}
\end{eqnarray*}
where $L_i=(v_i,\vec z)<0$ for given $\vec{z}\in \check{C}^-$, $1\leq i\leq n$.
\medskip

\noindent
(\mref{it:phidiv}) For a closed cone $C$ of rank $n$ in $\deff ^n$, let $\vec{z}\in \check{C}^-$ be as given in Item~(\mref{it:phiint}). Then the proof follows from Item~(\mref{it:phiint}):
\begin{eqnarray*}
 \Phi_n(C)&=&(-1)^n\int \cdots \int _C e^{x_1z_1+\cdots +x_nz_n}dx_1\cdots dx_n\\
&=&(-1)^n\sum_{i=1}^r \int \cdots \int _{C_i} e^{x_1z_1+\cdots +x_nz_n}dx_1\cdots dx_n\\
&=& \sum_{i=1}^r \Phi(C_i).
\end{eqnarray*}

In general, we can extend a minimal generating set $\{v_1, v_2, \cdots, v_k\}$ of a cone $C$ of rank $k$ in $\deff ^n$  to a basis $\{v_1, \cdots, v_k, v_{k+1}, \cdots, v_n\}$ of $\deff^n$. For a cone $D$ in the linear space $\deff v_1\oplus \cdots \oplus \deff v_k$, let $\overline{D}$ denote the cone in $\deff^n$ generated by $D$ and $v_{k+1}, \cdots, v_n$. Then clearly $\{\overline C_i\}$ is a simplicial subdivision in $\deff^n$ of $\overline C$.

Let $C_i$ be generated by $w^{(i)}_1, \cdots , w^{(i)}_k$ with $w^{(i)}_j=\sum_{\ell=1}^n b^{(i)}_{\ell j}v_\ell$ and let $M^{(i)}=(b^{(i)}_{\ell j})$. Then
$$A_{\overline C}=[A_C\,|B],\quad A_{\overline C_i}=[A_{C_i}|B],
$$
where $(v_{k+1},\cdots,k_n)=(e_1,\cdots,e_n)B$, and
$$A_{C_i}=A_{C}M^{(i)}, \quad
A_{\overline C_i}=A_{\overline C}\left [\begin {array}{cc}M^{(i)}&0\\0&I\end{array}\right ].
$$
Since the cone $\overline C$ has rank $n$, we have $\Phi_n(\overline C)=\sum\limits_{i=1}^r \Phi_n(\overline C_i)$. That is,
$$\frac {|\det(A_{\overline C})|}{L_{v_1}\cdots L_{v_k}L_{v_{k+1}}\cdots L_{v_n}}=\sum_{i=1}^r \frac {|\det(A_{\overline C_i})|}{L_{w^{(i)}_1}\cdots L_{w^{(i)}_k}L_{v_{k+1}}\cdots L_{v_n}}.
$$
Since $\det(A_{\overline C_i})=\det(A_{\overline C})\det(M^{(i)})$ and $w(C_i)=w(C)|\det(M^{(i)})|$, we reach the conclusion
$$\frac {w(C)}{L_{v_1}\cdots L_{v_k}}=\sum_{i=1}^r \frac {w(C_i)}{L_{w^{(i)}_1}\cdots L_{w^{(i)}_k}}.
$$
\end{proof}

Now let $C$ be a closed cone in $\deff^n$ and let $C=\{C_1,\cdots,C_r\}$ be a subdivision of $C$ into closed simplicial cones $C_1,\cdots,C_r$ in $\deff^n$. Define
\begin{equation} \Phi_n(C):=\sum_{i=1}^r \Phi_n(C_i).
\mlabel{eq:cphi}
\end{equation}
The value $\Phi_n(C)$ is well-defined because of the following lemma.
\begin{lemma}
For a closed cone $C$ in $\deff^n$, the value $\Phi_n(C)$ does not depend on the choice of the subdivision $C=\{C_1,\cdots,C_r\}$ of $C$ into closed simplicial cones in $\deff^n$.
\mlabel{lem:cphi}
\end{lemma}
\begin{proof}
Suppose $C=\{C'_1,\cdots,C'_{r'}\}$ is another subdivision of $C$ into closed simplicial cones in $\deff^n$. Let $C=\{C''_1,\cdots,C''_{r''}\}$ be a common refinement of the two subdivisions, which exists by Lemma~\mref{lem:ConeSubd}.(\mref{it:conesub}). Then by Lemma~\mref{lem:phi}.(\mref{it:phidiv}), we obtain
$$ \sum_{i=1}^r \Phi_n(C_i) = \sum_{k=1}^{r''} \Phi_n(C''_k) = \sum_{j=1}^{r'} \Phi_n(C'_j),$$
as needed.
\end{proof}

Thus we can extend the map $\Phi_n$ defined in Eq.~(\mref{eq:phi0}) on the set of closed simplicial cones to a linear map
\begin{equation}\Phi_n : \deff \cc_n(\deff) \to \cals(\deff), \quad
\Phi_n (C)=\sum_{i=1}^r \Phi_n (C_i),
\mlabel{eq:phin}
\end{equation}
where $\{C_i\}$ is taken to be any simplicial subdivision of $C$.
The linear maps $\Phi_n$ on $\deff\cc_n(\deff)$, $n\geq 1$, are compatible with the direct system $\{\deff\cc_n(\deff)\}_{n\geq 1}$ and can therefore be put together to build a linear map
\begin{equation}
\Phi: \deff \cc(\deff): =\deff (\cup_{n=1}^\infty \cc_n(\deff)) \to \cals(\deff).
\mlabel{eq:phi}
\end{equation}

Let $W_{\mathcal{C}}$ denote the subspace of $\deff\cc(\deff)$ generated by the following two types of elements:
\begin{enumerate}
\item
closed cones containing a linear $\deff$-subspace, and
\item
linear combinations $C-\sum_{i=1}^r C_i$ where $\{C_i\}$ is a subdivision in $\cc(\deff)$ of a closed cone $C\in \cc(\deff)$.
\end{enumerate}

\begin {prop}
\begin{enumerate}
\item
If a closed cone $C$ in $
\deff^n$ contains a line $\ell$, then there exist a
subspace $L$ of $C$ and a strongly convex closed cone $C'$ in $\lin^{\perp}
(L;\lin(C))$ that give the direct sum $C=L\dotplus C'$. Here $\lin^{\perp} (L;\lin(C))$ is the orthogonal complement of $L$ in $\lin(C)$, both taken as $\deff$-vector spaces.
\mlabel{it:line}
\item
We have $W_{\mathcal{C}}\subseteq \ker \Phi$.
\mlabel{it:phis}
\end{enumerate}
\mlabel{pp:phis}
\end{prop}

\begin{proof}
(\mref{it:line})
First assume that the line $\ell$ is contained in a proper face of $C$. Let $L$ be the maximal subspace $L\subset C$ that contains $\ell$. Let $C'$ be the projection of $C$ in $lin^{\perp} (L;lin(C))$. Then
$C=L+C'$. We know that $C'$ is strongly convex because $L$ is the maximal subspace in $C$.
\smallskip

Next assume that $\ell$ is not contained in any proper face of $C$. A generator $\vec u$ of $\ell$ gives rise to relative interior points $\vec u$ and $-\vec u$
   of $C$. Thus the projections of $C\cap
\{\vec v\, |\ (\vec v, \vec u)\ge 0\}$ and $C\cap \{\vec v \,|\ (\vec v,
\vec u)\le 0\}$ in $\lin^{\perp} (\ell;\lin(C))$  both coincide with $\lin^{\perp} (\ell;\lin(C))$. Notice that $\lin^{\perp} (\ell;\lin(C))\subset C$. Indeed, for any vector $v\in \lin^{\perp} (\ell;\lin(C))$, there exist $v'\in C$ and $a\in \RR _{\ge 0}$ such that $v'=au+v$. So $v=a(-u)+v'\in C$ by convexity. Therefore
$$C\supseteq
(\RR _{\ge 0}\vec u +\lin^{\perp} (\ell;\lin(C)))\cup (\RR _{\ge 0}(-\vec u
) +\lin^{\perp} (l;\lin(C)))$$
which implies that
$$C=\lin(\vec u) +\lin^{\perp} (\ell;\lin(C))$$
 is a linear
subspace, proving the claim.
\medskip

\noindent
(\mref{it:phis}) Because of Lemma~\mref{lem:phi}.(\mref{it:phidiv}), we only need to prove that $\Phi (C)=0$ if $C $ contains a line.

First consider the case when $C$ itself is a one-dimensional subspace.
Let $C=\deff_{\ge 0}u \cup \deff_{\ge 0}(-u)$. So
$$\Phi (C)=\frac {w(u)}{L_{u}}+\frac {w(-u)}{L_{-u}}=\frac {w(u)}{L_{u}}+\frac {w(  u)}{-L_{  u}}=0.
$$

Next consider the case when $C$ is a non-zero linear space. Take any basis $\{v_1, \cdots , v_k\}$ of $C$.
The family of cones $\{C_{\varepsilon_1\varepsilon_2\cdots\varepsilon_k} :=\ccone{\varepsilon_1 v_1,\varepsilon_2 v_2,\cdots ,\varepsilon_k v_k}\ | \varepsilon _i=\pm 1\}$ provides  a simplicial subdivision of $C$  and
$$w(\varepsilon_1 v_1, \varepsilon_2 v_2, \cdots, \varepsilon_k v_k)=w(v_1, v_2, \cdots, v_k).
$$
Thus
\begin{eqnarray*}
\Phi(C)&=&\sum_{\vep_i=\pm 1, 1\leq i\leq k} \Phi(C_{\varepsilon_1\varepsilon_2\cdots\varepsilon_k} ) \\ &=&\sum_{\vep_i=\pm 1, 1\leqq i\leq k} \frac{w(\vep_1 v_1,\cdots,\vep_k v_k)}{L_{\vep_1 v_1}\cdots L_{\vep_k v_k}} \\
&=& w(v_1,\cdots,v_k)\sum_{  1\leq i\leq k} \left (\frac{1}{L_{v_i}}+\frac{1}{L_{-v_i}}\right)\\
&=&w(v_1,   \cdots, v_k)\sum_{  1\leq i\leq k} (\frac {1}{L_{  v_i}}-\frac {1}{L_{  v_i}})\\
&=&0.
\end{eqnarray*}

Finally consider the case when $C$ is a cone that contains a line. By Proposition~\mref{pp:phis}.(\mref{it:line}) we have
$C=L \dotplus C'$ where $L$ is a linear subspace and $C'$ is a strongly convex cone.
Given a basis $\{v_1, \cdots , v_k\}$ of $L$,
the set $\{C_{\varepsilon_1,\varepsilon_2,\cdots,\varepsilon_k} :=\ccone{\varepsilon_1 v_1,\varepsilon_2 v_2,\cdots, \varepsilon_k v_k}+C'\ | \varepsilon _i=\pm 1\}$ provides  a subdivision of $C$. As in the case of a linear subspace, we have  on the one hand
$$w(C_{\varepsilon_1, \varepsilon_2 , \cdots, \varepsilon_k })=w(C_{1,1,\cdots ,1}),
$$
and  on the other hand
$$\Phi (C_{-1, \varepsilon_2 , \cdots, \varepsilon_k })=-\Phi (C_{1, \varepsilon_2 , \cdots, \varepsilon_k }).
$$
Consequently
$$\Phi(C)=\sum _{\varepsilon_1,\varepsilon_2,\cdots,\varepsilon_k}\Phi (C_{\varepsilon_1,\varepsilon_2,\cdots,\varepsilon_k})
=\sum _{\varepsilon_2,\cdots,\varepsilon_k}\Phi (C_{1,\varepsilon_2, \cdots, \varepsilon_k})+\Phi (C_{-1,\varepsilon_2, \cdots, \varepsilon_k})=0.$$
\end{proof}

\subsection {Subdivisions of cones and simple fractions}
\mlabel{subsec:sisom}

We next show that all relations among simple fractions are determined by those coming from subdivisions of the corresponding cones. As a preparation, we give some properties of cones and fractions.

\begin{lemma}
Let $\{C_i\}$ be a set of closed cones in $\deff^n_{\geq 0}$ that span the same linear subspace of $\deff^n$ and meet with each other only along faces. Then the set $\{\Phi(C_i)\}$ of fractions is linearly independent.
\mlabel{lem:gencone}
\end{lemma}

\begin{proof}
We proceed by induction on the dimension of the closed cones $C_i$. If the dimension is $1$, then the set $\{C_i\}$ can contain only one element $\ccone{v}$ where $v$ is a nonzero vector. Thus, $\Phi(C_i)$ is a nonzero multiple of $1/L$ where $L$ is a nonzero linear form and the lemma is proved.

Assume that the lemma has been proved when the dimension of $C_i$ is $k\geq 1$ and consider a set $\{C_i\}$ of closed cones with dimension $k+1$ that satisfy the conditions in the lemma. Suppose $\{\Phi(C_i)\}_i$ is a linearly dependent family. Taking a subset of $\{\Phi(C_i)\}_i$ if necessary, we can assume that there are nonzero $a_i\in \deff, 1\leq i\leq r,$ such that
$$ f:=\sum_{i=1}^r a_i\Phi(C_i)=0.$$

Since $\cup_{i=1}^r C_i$ is contained in $\deff^n_{\geq 0}$ and hence has a topological boundary, we can also assume that $C_1=\ccone{v_1, \cdots, v_k}$ contains part of the boundary of $\cup_{i=1}^r C_i$. Then one of its facets, say the one spanned by $v_2, \cdots, v_k$, is not contained in any other cone. We can therefore rewrite $f=\sum\limits_{i=1}^r c_i\Phi(C_i)$ as
$$f=\frac {a_1}{L_1\cdots L_k}+\sum_{i=2}^r\frac{a_i}{L_{i1}\cdots L_{ik}},
$$
such that $L_2\cdots L_k\nmid L_{i1}\cdots L_{ik}, 2\leq i\leq r$. Since all the $C_i$s span the same linear space, for each $2\leq i\leq r$, we have $L_1=\sum_{j=1}^k c_{ij} L_{ij}$ so that
$$ \frac{L_1}{L_{i1}\cdots L_{ik}}=\sum_{j=1}^k \frac{c_{ij}}{L_{i1}\cdots \check{L}_{ij} \cdots L_{ik}},$$
where $\check{L}_{ij}$ means the term $L_{ij}$ is deleted.
Thus from $f=0$ we deduce that
$$0=L_1f= \frac{a_1}{L_2\cdots L_k}+ \sum_{i=2}^k a_i \sum_{j=1}^k \frac{c_{ij}}{L_{i1}\cdots \check{L}_{ij} \cdots L_{ik}},$$
which  is a linear combination of simple fractions of degree $k-1$. Since $L_2\cdots L_k$ does not divide any of the other forms, the coefficient of $1/(L_2\cdots L_k)$ is $a_1$ which is not zero by assumption. Furthermore the cones corresponding to the fractions are faces of $\{C_i\}$ and hence meet each other along  faces. Thus by the induction hypothesis, all the coefficients are zero, which  is a contradiction. This proves Lemma~\mref{lem:gencone}.
\end{proof}

Let $\deff(\vec{z})=\deff(\{z_i\}_{i\geq 1})$ be the field of fractions in the variables $\{z_i\,|\,1\leq i<\infty\}$ with coefficients in $\deff$. An element $f\in \deff(\{z_i\}_{i\geq 1})$ is called {\bf homogenous of degree $k$} if $f(\{tz_i\})=t^{-k}f(\{z_i\})$ for a nonzero scalar $t$.

\begin{lemma}
Let $f(\vec{z})\in \deff(\vec{z})$ be of the form $\sum_{k=1}^\infty f_k(\vec{z})$, where $f_k$ is homogeneous of degree $k$. If $f=0$, then $f_k=0, k\geq 1$.
\mlabel{lem:homog}
\end{lemma}
\begin{proof}
For any given value $\vec{z}_0:=(z_{0,1},\cdots,z_{0,n},\cdots)$ of $\vec{z}$, consider the substitutions $z_i=z_{0,i} t$, where $t$ is a nonzero scalar. Then we obtain $f(\vec z)=\sum_k
f_k(\vec z_0)t^{-k}$ so that every coefficient $f_k(\vec{z}_0)$ has to be $0$. Thus $f_k(\vec{z})=0$.
\end{proof}

\begin{lemma}
\begin{enumerate}
\item
Let $k\geq 1$ be given. Let $g,h\in \deff(\vec{z})$ be linear combinations of simple fractions of homogeneous degree $k$ such that the linear factors in $g$ are in the linear span of $z_1,\cdots,z_k$ only, while at least one linear factor in each simple fraction in $h$ is not in the linear span of $z_1,\cdots,z_k$. If $g=h$, then $g=h=0$.
\mlabel{it:gh}
\item Let $f$ be in $\cals(\deff)$ and let $\frac{1}{L_1\cdots L_k}$ be a simple fraction in $f$. Let $G$ be the summand of $f$ consisting of simple fractions whose linear factors are in the linear span of $L_1,\cdots,L_k$. If $f=0$, then $G=0$.
\mlabel{it:comb}
\end{enumerate}
\mlabel{lem:gh}
\end{lemma}
\begin{proof}
(\mref{it:gh}) Let
$$h=\sum_{i=1}^r \frac{a_i}{L_{i1}\cdots L_{ik}}.$$
Then for each $1\leq i\leq r$, at least one $L_{ij}$ is not in the linear span of $z_1,\cdots,z_k$. For an positive integer $n$ we set $[n]= \{1,\cdots, n\}$. Let $J\subseteq [r]\times [k]$ be the set of indices $(i,j)\in [r]\times [k]$ such that $L_{ij}$ do not lie in the linear span of $z_1,\cdots,z_k$. For $(i,j)\in J$, we write
$$ L_{ij}=L_{ij}' + L_{ij}'',$$
where $L_{ij}'$ lies in the linear span of $z_1,\cdots,z_k$ and $L_{ij}''$  is a nonzero form in the linear span of $z_{k+1},\cdots,z_n$. Here $n$ is the largest index such that $z_n$ appears in $L_{ij}$. Thus the product $\prod_{i,j}L_{ij}''$ is a nonzero polynomial so that  there is an evaluation  $z_i=c_i, c_i\in \deff, k+1\leq i\leq n,$ such that
$$\ell_{ij}:=L_{ij}''(c_{k+1},\cdots,c_n)\neq 0, (i,j)\in J.$$

Take the substitution $z_i=x_it, 1\leq i\le k$ and $ z_i=c_i, k+1\le k\le n$ in $g$ and $h$, where $x_i$ and $t$ are variables. On one hand we have $g(x_1t,\cdots,x_kt)=t^{-k} g(x_1,\cdots,x_k).$ So $g$ has a pole at $t=0$ of order $k$ unless $g=0$. On the other hand, for $(i,j)\in J$, since $L_{ij}''$ is nonzero in $\deff$ for this substitution, $1/L_{ij}(x_1t,\cdots,x_kt,c_{k+1},\cdots,c_n)$   does not have a pole at $t=0$. The order of $t=0$ as pole of $h(x_1t,\cdots,x_kt,c_{k+1},\cdots,c_n)$ is therefore at most $k-1$. Hence we must have $g=0$ and $h=0$.
\smallskip

\noindent
(\mref{it:comb})
Fix a simple fraction $\frac{1}{L_1\cdots L_k}$ in $f$. Let $f=G+H$, where $G$ is the linear combinations of simple fractions
whose linear forms are linear combinations of $L_1, \cdots, L_k$ only, and $H$ is the linear
combination of simple fractions with at least one linear form that is not any linear combination of $L_1,\cdots,L_k$.

Expand $L_1,\cdots,L_k$ to a basis $L_1,\cdots,L_k,L_{k+1},\cdots,L_n$ of the linear subspace generated by the linear forms that appear in $f$. Expand $L_1,\cdots,L_n$ further to a system $\{L_i\}_{i\geq 1}$ of linear forms that form a basis of the linear span of $\{z_i\}_{i\geq 1}$. Then the linear map $L_i\mapsto z_i, 1\leq i<\infty$ induces an algebraic automorphism on $\deff(\vec{z})$. Under this automorphism, the element $G$ (resp. $H$) above is sent to a $g$ (resp. an $-h$) in the first part of the  lemma . Thus $f=0$ means $g=h$. By  the first part of the lemma, $g=h=0$ and hence $G=H=0$.
\end{proof}

\begin {defn} {\rm
\begin{enumerate}
\item
Consider a fraction $\frac {a}{L_1\cdots L_k}$, where $a\in \ZZ$ and $L_1=\sum_{j=1}^n a_{1j}z_j,\cdots, L_k=\sum _{j=1}^n a_{kj}z_j \in \ZZ [\vec{z}]$. If $v_1:=(a_{11},\cdots, a_{1n}), \cdots, v_k:=(a_{k1},\cdots, a_{kn})$ is part of a basis of $\ZZ^n$, then the fraction is called a {\bf smooth simple fraction}. The linear subspace of $\QQ(\vec{z})$ spanned by the smooth simple fractions is denoted by $\cals_\cm(\QQ)$.
\item
Let $W_\cm$ denote the subspace of $\QQ\mcc(\QQ)$ generated by the following two types of elements:
\begin{enumerate}
\item
smooth closed cones containing a linear $\QQ$-subspace, and
\item
linear combinations $C-\sum_{i=1}^r C_i$, where $\{C_i\}$ is a smooth subdivision of a closed smooth cone $C$.
\end{enumerate}
\end{enumerate}
}
\end{defn}

\begin {lemma} We have $ \cals(\QQ) = \cals_\cm(\QQ),$ where as before $ \cals(\QQ)$ stands for the space of simple fractions with rational coefficients.
\mlabel{lem:smfr}
\end {lemma}
\begin{proof}
Let $1/(L_1\cdots L_k)$ be a simple fraction in $\cals(\QQ)$ with $L_i=\sum_{j=1}^n c_{ij}z_j\in \QQ[z_1,\cdots,z_n]$. Let $v_i=\sum_{j=1}^n c_{ij} e_j$. By Proposition~\mref{pp:ratsm}.(\mref{it:ratsm}), the rational cone $\ccone{v_1,\cdots,v_k}$ admits a subdivision $\{C_i\}$ consisting of smooth cones. By Proposition~\mref{pp:phis}.(\mref{it:phis}), we have
$$ \frac{1}{L_1\cdots L_k} =\Phi\left(\frac{1}{w(v_1,\cdots,v_k)}\ccone{v_1,\cdots,v_k}\right) = \sum_i \frac{1}{w(v_1,\cdots,v_k)}\Phi(C_i)$$
with a multiple of $\Phi(C_i)$ smooth, as needed.
\end{proof}

Thus the map $\Phi: \QQ\cc(\QQ)\to \cals(\QQ)$ in Eq.~(\mref{eq:cphi}) restricts to a map
$$ \Phi_\cm: \QQ\mcc(\QQ)\to \cals_\cm(\QQ).$$

\begin{theorem}\label{thm:Phibijective}
Let $W_{\mathcal{C}}$ be   defined as above Proposition~\mref{pp:phis}.
Then
$\ker \Phi = W_{\mathcal{C}}$ and $\ker \Phi_\cm =W_\cm$. Thus we have linear bijections
\begin{equation}
\overline{\Phi}: \deff \overline{\cc}(\deff):=\deff \cc(\deff)/W_{\mathcal{C}}\cong \cals(\deff)\ \text{ and }
\overline{\Phi}_\cm: \QQ \overline{\mcc}(\QQ):=\QQ \cc(\QQ) /W_\cm\cong \cals_\cm(\QQ).
\mlabel{eq:phibij}
\end{equation}
\mlabel{thm:cfbij}
\end{theorem}

\begin{proof}
The surjectivity of the maps $\Phi$ and $\Phi_\cm$ follows from  the definitions of $\cals(\QQ)$ and $\cals_\cm(\QQ)$ combined with  Lemma~\mref{lem:smfr}. By Proposition~\mref{pp:phis}.(\mref{it:phis}), we have $W_{\mathcal{C}}\subseteq \ker \Phi$ and $W_{\mathcal{SC}}\subseteq W_{\mathcal{C}}(\QQ)\subseteq \ker\Phi_\cm$. So we only need to prove $\ker \Phi \subseteq W_{\mathcal{C}}$ and $\ker \Phi_\cm\subseteq W_{\mathcal{SC}}$.
We first prove the first inclusion.

Let $\sum_{i=1}^r a_i C_i$ be in $\ker \Phi$. First, we can assume that $\cup C_i$ has boundary. This is because if $\cup C_i$ is not the whole space, then it has a boundary. Otherwise, fix a point $v_0$, for all cones $\{C_j'\}\subset \{C_i\}$ containing $v_0$ as an interior point, then modulo subdivision therefore modulo $W_{\mathcal{C}}$, we can assume that $C_j'$ simplicial. If $C_j'=\ccone{v_1,\cdots, v_k}$, then modulo $W_{\mathcal{C}}$,
$$C_j'\sim \sum \ccone{\epsilon _1v_1, \cdots, \epsilon _kv_k},$$
where $\epsilon_i=\pm 1$ and the summation is taken over all possible $\epsilon _i$ except when all $\epsilon _i=1$. Now, $v_0$ is not an interior point of any resulting cones, so the union of resulting cones has a boundary.

Then
$$0=f(\vec{z}):= \Phi\left(\sum_{i=1}^r a_iC_i\right) = \sum_{i=1}^r a_i\Phi(C_i)\in \RR(\vec{z})$$
as a linear combination of monic simple fractions $\Phi(C_i)$. We  prove that $\sum_{i=1}^r a_i C_i$ lies in $W_{\mathcal{C}}$ by reducing this statement to Lemma~\mref{lem:gencone} by means of  the following reduction steps.

By Lemma~\mref{lem:homog}, we may assume  that $f(\vec{z})$ is homogeneous of degree $-k$.

By choosing a simple fraction $\frac{1}{L_1\cdots L_k}$ in $f$ whose linear forms span a minimal linear subspace and then applying Lemma~\mref{lem:gh}.({it:comb}), we may further assume that the linear forms of each simple fractions in $f$ span the same linear subspace.

Let $L_1,\cdots, L_k, L_{k+1}, \cdots , L_m$ be all the linear forms in $f$. Let
$$\{ C_j\,|\, 1\leq j\leq \ell\}$$
be the set of cones corresponding to the simple fractions in $f$, that is,
$$ f =\sum_{j=1}^l a_j\Phi(C_j).$$
On the grounds of our assumptions on the simple fractions, we conclude that the $C_j$'s span the same linear subspace of $\RR^n$.

Choose the simplicial subdivision $\{C_{ij}\}$ of $C_i, 1\leq i\leq r,$ as in Lemma~\mref{lem:ConeSubd}.(\mref{it:conesub}). Then
$$ \sum_{i=1}^r a_i C_i - \sum_{i,j}a_i C_{ij}$$
lies in $W_{\mathcal{C}}$ and we can write
$$\sum_{i,j} a_i C_{ij}=\sum_{\ell} b_\ell D_\ell,$$
where $D_\ell$ are distinct simplicial cones that meet only at faces.
Since
$$ 0=\Phi\left(\sum_{i=1}^r a_i C_i\right) = \Phi\left(\sum_{ij} a_i C_{ij} \right) =\Phi\left(\sum_{\ell} b_\ell D_\ell\right)=\sum_{\ell} b_\ell \Phi(D_\ell),$$
by Lemma~\mref{lem:gencone}, all the coefficients $b_\ell$ are zero. Thus
$$\sum_{i,j}a_i C_{ij}=\sum_{\ell} b_\ell D_\ell=0.$$
Therefore,
$$ \sum_{i=1}^r a_i C_i = \sum_{i=1}^r a_i C_i - \sum_{i,j}a_i C_{ij}$$
lies  in $W_{\mathcal{C}}$.
\smallskip

The proof of the second inclusion $\ker \Phi_\cm\subseteq W_\cm$ is the same as the proof of the first inclusion up to that   Lemma~\mref{lem:ConeSubd}.(\mref{it:conesub}) is replaced  by
Lemma~\mref{lem:ConeSubd}.(\mref{it:smoothsub}).
\end{proof}

\section{Decorated cones and pure fractions}
\mlabel{sec:dcone}
We next generalize the geometric interpretation of linear relations of simple fractions as subdivision of cones to the fractions with multiplicity for the linear forms. For this purpose, we need to generalize the notion of smooth cone to smooth decorated cones which involve multiplicity encoded in the decoration.

\subsection{Decorated closed cones}
\mlabel{ss:fcc}

\begin{defn} {\rm
\begin{enumerate}
\item
Let $C=\ccone{v_1,\cdots,v_k}$ be a smooth cone in the first orthant with its (unique) primary generating set $\{v_1,\cdots,v_k\}$ and let $s_1,\cdots,s_k$ be in $\ZZ_{\geq 1}$. We call the monomial $[v_1]^{s_1}\cdots [v_k]^{s_k}$ a {\bf decorated smooth cone}, $s_1+\cdots +s_k$ the {\bf weight} of the decorated cone and   $\ccone{v_1,\cdots,v_k}$ the underlying {\bf geometric cone} of the decorated cone which we can also denote by $[v_1]\cdots [v_k]$.
\item
The set of decorated smooth cones is denoted by $\dsmc$, regarded as a subset of the polynomial algebra $\QQ[\{[v]\in \QQ^\infty\}]$:
$$\dsmc:=\left\{[v_1]^{s_1}\cdots[v_k]^{s_k}\in\QQ\left [\{[v]\in \QQ^\infty\}\right] \,\Big|\, \begin{array}{l} \{v_1,\cdots,v_k\} \subseteq \ZZ^n \text{ is part of a basis of } \ZZ^n\\ s_1,\cdots,s_k\geq 1\end{array}\right\}.
$$
\item
For $i\geq 1$, define the {\bf conical derivation} in direction $e_i$ to be the linear operator
$$\delta_i: \QQ \dsmc \to \QQ \dsmc, \quad \delta_i( [v_1]^{s_1}\cdots [v_k]^{s_k}):=\sum _j s_j \ (e_i^*, v_j)[v_1]^{s_1}\cdots [v_j]^{s_j+1}\cdots [v_k]^{s_k},$$
where $\{e_1^*, e_2^*,\cdots \}$ is the dual basis to $\{e_1, e_2,\cdots \}$, and $(e_i^*, v_j)$ is the pairing between $e_i^*$ and $v_j$. Here $\QQ \dsmc$ is the $\QQ$-linear span of the set $\dsmc$ of smooth decorated closed cones.
\end{enumerate}
}
\end{defn}

\begin{remark}
{\rm
\begin{enumerate}
\item
The notion $[v_1]^{s_1}\cdots [v_k]^{s_k}$ is well-defined since the primitive generating set is unique.
\item
Since we will only consider decorated smooth cones, we will often suppress smooth from the notations.
\item
The term conical derivation is justified because of the following fact.
\end{enumerate}
}
\end{remark}

\begin{prop}
Let $V:=[v_1]^{s_1}\cdots [v_k]^{s_k}$ be a decorated smooth cone.
\begin{enumerate}
\item
The operator $\delta_i$ can be equivalently defined by
\begin{enumerate}
\item[(i)] $\delta_i([v])=(e_i^*,v)[v]^2$ for a smooth vector $v\in \ZZ^n$, and
\item[(ii)]
{\rm (Weak Leibniz Rule)} If $V$ is the product of decorated cones $V_1$ and $V_2$, then $\delta_i(V)=\delta_i(V_1)V_2+V_1\delta_i(V_2)$.
\end{enumerate}
\mlabel{it:leib}
\item
Let $\{v_i^*=\sum_j
c_{ij} e_j^*\}_i$ be a dual basis to $\{v_i\}_i$ in the sense that $(v_i, v_j^*)=\delta
_{ij}, 1\leq i, j\leq k$.
Define $\delta_{v_i^*}=\sum_{j} c_{ij} \delta_j.$ We have
\begin{equation}
[v_1]^{s_1}\cdots [v_k]^{s_k} =\frac
1{(s_1-1)!\cdots (s_k -1)!}\delta_{v_1^*}^{s_1-1}\cdots
\delta_{v_k^*}^{s_k-1}([v_1]\cdots [v_k]). \mlabel{eq:DiffCone}
\end{equation}
\mlabel{it:back}
\end{enumerate}
\mlabel{prop:DerCone}
\end{prop}

\begin{remark}
{\rm
The product of two decorated smooth cones is not necessarily a decorated smooth cones so that we call "weak Leibniz rule" the above product rule which only applies for decorated smooth cones that can be factored into a product of two decorated smooth cones.
}
\end{remark}

\begin{proof}
(\mref{it:leib}) Since $\dsmc$ is multiplicatively generated by smooth vectors, there is unique operator $\delta$ satisfying the two conditions.  On the other hand, $\delta_i$ satisfies the first condition by definition. Further note that $V_1$ and $V_2$ must be of the form
$V_1=[v_1]^{a_1}\cdots [v_k]^{a_k}$ and $V_2=[v_1]^{b_1}\cdots [v_k]^{b_k}$ with $a_i+b_i=s_i$. Then $\delta_i$ also satisfies the second condition. Thus $\delta$ and $\delta_i$ must be the same.

\smallskip

\noindent
(\mref{it:back})
This will be proved by induction on $m:=\vert s\vert-k$, $s_i\geq 1, 1\leq i\leq k, k\geq 1$, with the case $m=0$ being trivial. Assume that Eq.~(\mref{eq:DiffCone}) has been proved for all fractions with $m=n\geq 1$ and consider a decorated cone
$[v_1]^{s_1}\cdots [v_k]^{s_k}$ with $m=n+1$.

Let $s_r$ be the first $s_i, 1\leq i\leq m$ with $s_i>1$. Then for $t_r=s_r-1$ and $t_i=s_i, i\neq r$, we have
$\vert t\vert -k=n$, so that  the induction assumption yields
\begin{equation}
[v_1]^{t_1}\cdots \cdots [v_k]^{t_k}= \frac{1}{(t_1-1)!\cdots (t_k -1)!}\delta_{v_1^*}^{t_1-1} \cdots
\delta_{v_k^*}^{t_k-1}([v_1]\cdots [v_k]).
\mlabel{eq:lind}
\end{equation}

Let $v_i$ and $v_j^*$, $1\leq i,j\leq k$, be  defined as in the proposition. Then for the column vectors $\vec{e}=(e_1,\cdots,e_n)^T$, $\vec{v}=(v_1,\cdots,v_k)^T$ and $\vec{v}^*=(v_1^*,\cdots,v_k^*)^T$, we have
$$\vec{v}=A \vec{e}, \quad \vec{v}^* = C \vec{e}$$
for $A=(a_{ij}), C=(c_{ij})\in M_{k\times n}(\RR)$.
By  duality, $AC^T=I_k$, that is,
$\sum_{\ell=1}^n a_{i\ell}c_{j\ell}=\delta_{ij}, 1\leq i,j\leq k,$ from which it follows that
\begin{eqnarray*}
\delta_{{v^*_r}}([v_1]^{t_1}\cdots [v_k]^{t_k})&=&
\big(\sum_{j=1}^n c_{rj}\delta_{j}\big)([v_1]^{t_1}\cdots [v_k]^{t_k})
\\&=& \sum_{j=1}^n c_{rj} \sum_{i=1}^n t_ia_{i,j} [v_1]^{t_1}\cdots [v_i]^{t_i+1}\cdots [v_k]^{t_k} \\
&=& \sum_{i=1}^n t_i \big(\sum_{j=1}^n a_{i,j}c_{rj}\big) [v_1]^{t_1}\cdots [v_i]^{t_i+1}\cdots [v_k]^{t_k} \\
&=& t_r [v_1]^{t_1}\cdots [v_r]^{t_r+1}\cdots [v_k]^{t_k}.
\end{eqnarray*}
Combining this with the induction hypothesis in Eq.~(\mref{eq:lind}), we obtain
{\allowdisplaybreaks
\begin{eqnarray*}
[v_1]^{s_1}\cdots  [v_k]^{s_k} &=&
[v_1]^{t_1}\cdots [v_r]^{t_r+1} \cdots [v_k]^{t_k} \\ &=& \frac 1{(t_1-1)!\cdots t_r! \cdots (t_k -1)!} \delta_{v_1^*}^{t_1-1}\cdots \delta_{v_r^*}^{t_r}\cdots
\delta_{v_k^*}^{t_k-1}([v_1]\cdots  [v_k])\\
&=& \frac
1{(s_1-1)!\cdots  (s_k -1)!} \delta_{v_1^*}^{s_1-1}\cdots
\delta_{v_k^*}^{s_k-1}([v_1]\cdots  [v_k]).
\end{eqnarray*}
}
This completes the induction.
\end{proof}

We generalize subdivisions of geometric cones to algebraic subdivisions of decorated cones.

\begin {defn}
{\rm
\begin {enumerate}
\item An {\bf algebraic subdivision} of  a smooth cone $[v_1]\cdots [v_k]$ is an element
$$\sum _i [v_{i1}]\cdots [v_{ik}] \in \QQ \dsmc,$$
where $\{[v_{i1} ]\cdots [v_{ik}]\}_i$ is a smooth subdivision of the cone $\ccone{v_1,\cdots,v_k}$ in $\mcc$.
\item An {\bf algebraic subdivision} of  a decorated smooth cone $[v_1]^{s_1}\cdots [v_k]^{s_k}$  is an element
$$\sum _i \frac
1{(s_1-1)!\cdots (s_k -1)!}\delta_{v_1^*}^{s_1-1}\cdots
\delta_{v_k^*}^{s_k-1}([v_{i1}]\cdots [v_{ik}]) \in \QQ \dsmc,
$$
where $\{[v_{i1}]\cdots [v_{ik}]\}_i$ is a closed smooth subdivision of the cone $[v_1]\cdots [v_k]$. We will use the notation
$$ [v_1]^{s_1}\cdots [v_k]^{s_k} \prec \sum _i \frac
1{(s_1-1)!\cdots (s_k -1)!}\delta_{v_1^*}^{s_1-1}\cdots
\delta_{v_k^*}^{s_k-1}([v_{i1}]\cdots [v_{ik}])
$$
to denote such an algebraic subdivision.
\end {enumerate}
}
\end{defn}

\begin{exam} {\rm
From the algebraic subdivision
$$[e_1][e_2] \prec [e_1][e_1+e_2]+[e_2][e_1+e_2],
$$
we obtain the algebraic subdivision
\begin{eqnarray*}
&&[e_1]^2[e_2]=\delta_{e_1^*}([e_1][e_2]) \\
&\prec& \delta_{e_1^*}([e_1][e_1+e_2]+[e_2][e_1+e_2]) =[e_1]^2[e_1+e_2]+[e_1][e_1+e_2]^2+[e_2][e_1+e_2]^2.
\end{eqnarray*}
}
\mlabel{ex:decsub}
\end{exam}

\subsection {Subdivision of decorated cones and pure fractions}

\begin{defn}
{\rm
\begin{enumerate}
\item
For a smooth simple fraction $\frac{1}{L_1\cdots L_k}$ and $s_1,\cdots,s_k\geq 1$, the fraction $\frac{1}{L_1^{s_1}\cdots L_k^{s_k}}$ is called a {\bf smooth pure fraction.}
The integer $\sum_{i=1}^k s_i$ is called the {\bf degree} of the fraction. The linear subspace of $\QQ(\{z_n\}_{n\geq 1})$ spanned by smooth pure fractions is denoted by $\calf _\cm(\QQ)$.
\item
Denote
$$\partial_i=-\frac{\partial}{\partial z_i}: \calf_\cm(\QQ)\to \calf_\cm(\QQ), \quad i\geq 1.$$
\item
Define
\begin{equation}
\Phi_{\dm}:= \QQ \dsmc \to \calf_\cm(\QQ), [v_1]^{s_1}\cdots [v_k]^{s_k} \mapsto \frac{w(v_1,\cdots,v_k)}{L_1^{s_1}\cdots L_k^{s_k}},
\mlabel{eq:phidsm}
\end{equation}
where $w(v_1,\cdots,v_k)$ is defined in Eq.~(\mref{eq:phi0}).
\end{enumerate}
}
\end{defn}

The following proposition follows from straightforward computations using Proposition~\mref{prop:DerCone}.

\begin{prop}
\begin{enumerate}
\item
For $i\in I$, we have $\Phi_\dm \circ \delta_i = \partial_i\circ \Phi_\dm$.
\mlabel{it:DPhi}
\item
Let a smooth pure fraction $\frac{1}{L_1^{s_1}\cdots L_k^{s_k}}$ be given. Let $\{L_i^*=\sum_j
c_{ij} z_j^*\}_i$ be dual to $\{L_i\}_i$ in the sense that $(L_i, L_j^*)=\delta
_{ij}, 1\leq i, j\leq k$. Define $\partial_{L_i^*}=\sum_{j} c_{ij} \partial_j.$ Then we have
\begin{equation}
\frac{1}{L_1^{s_1}\cdots L_k^{s_k}}
= \frac
1{(s_1-1)!\cdots (s_k -1)!}\partial_{L_1^*}^{s_1-1}\cdots
\partial_{L_k^*}^{s_k-1}\frac{1}{L_1\cdots L_k}. \mlabel{eq:diffrac}
\end{equation}
\mlabel{it:DerFrac}
\end{enumerate}
\mlabel{prop:DerFrac}
\end{prop}

\begin{lemma}
Let $\{C_i\}_i$ be a set of decorated closed smooth cones in $\QQ^n$ such that the linear forms in every cone span the same linear subspace of $\QQ^n$ and the underlying geometric cones meet only along faces. Then the set $\{\Phi_{\dm}(C_i)\}_i$ of fractions is linearly independent.
\mlabel{lem:DecoratedCone}
\end{lemma}

\begin {proof} We just need to prove that   a contradiction follows from any relation
\begin{equation}
\sum_{i=1} ^r a_i \Phi_{\dm}(C_i)=0
\mlabel{eq:contrel}
\end{equation}
with $0\neq a_i\in \QQ$ and $C_i=[v_{i1}]^{s_{i1}}\cdots [v_{ik}]^{s_{ik}}$   closed smooth cones with the conditions in the lemma. By Lemma \mref{lem:homog}, we can assume that for each $1\leq i\leq r$, the weight $\vert s_i\vert =s_{i1}+\cdots +s_{ik}$ is the same. We next proceed by induction on $s:=\vert s_i \vert$. So $s\geq k$.

If $s=k$, then all the edges of the cones have decoration 1. Then by Lemma \mref {lem:gencone} we must have $a_i=0$ for any index $i$, leading to the expected contradiction. Assume that a contradiction arises for any relation in Eq.~(\mref{eq:contrel}) with $s= n\geq k$ and consider such a relation with $s=n+1$. In this case, at least one edge, say $[v_1]$ in some decorated cone, has decoration greater than one.

Let $r_1$ be the maximal decoration of $[v_1]$ in all the cones. We split the family of cones into three disjoint sets. Let $C_1, \cdots, C_m$ be all the cones with edge $[v_1]$ decorated by $r_1$. Let $C_{m+1},\cdots,C_{m+\ell}$ be all the cones, if any, with $[v_1]$ decorated by a positive power less than $r_1$. Let $C_{m+\ell+1},\cdots,C_r$ be all the cones, if any, that do not contain $[v_1]$ in their spanning set. Then
$$L_{v_1}\sum_{i=1}^r a_i \Phi_{\dm}(C_i)=L_{v_1}\sum_{i=1}^m a_i \Phi_{\dm}(C_i)+L_{v_1}\sum _{i=m+1}^{m+\ell} a_i \Phi_{\dm}(C_i) +L_{v_1}\sum _{i=m+\ell+1}^r a_i \Phi_{\dm}(C_i).
$$
For any $m+1\leq i\leq m+\ell$, the power of $1/L_{v_1}$ in $L_{v_1}\Phi_{\dm}(C_i)$ is less than $r_1-1$. For $m+\ell+1\leq i\leq r$, $1/L_{v_1}$ does not appear as a linear form of $\Phi_{\dm}(C_i)$. Using the assumption of the proposition, we  write $L_{\nu_1}$ as a linear combination of the generators $L_{v^i_1},\cdots,L_{v^i_k}  $ of $C_i$:
$$L_{v_1}=a_{i1}L_{v_{i1}}+\cdots +a_{ik}L_{v_{ik}}.
$$
Thus $L_{v_1}\Phi_{\dm}(C_i)=\sum\limits_{j=1}^k \frac{a_{i1}}{L_{v_{i1}}^{s_{i1}}\cdots L_{v_{ij}}^{s_{ij}-1}\cdots L_{v_{ik}}^{s_{ik}}}$ is a linear combination of fractions that do not contain $L_{v_1}$ as a linear form. In summary, each monomial in $\sum\limits_{i=m+1}^r a_i L_{v_1}\Phi_{\dm}(C_i)$ has its power of $1/L_{v_1}$ less than $r_1-1$ so that no such monomial can   cancel  with any monomial in $L_{v_1}\sum\limits_{i=1}^m a_i\Phi_{\dm}(C_i)$.
Then from
$\sum\limits_{i=1}^r a_i\Phi_{\dm}(C_i)=0$ and thus  $L_{v_1}\sum\limits_{i=1}^r a_i\Phi_{\dm}(C_i)=0$.

In the equality
$$L_{v_1}\sum\limits_{i=1}^r a_i\Phi_{\dm}(C_i)=0,$$
the monomials from decorated cones $C_1, \cdots,C_m$ have non-zero coefficients $a_1, \cdots, a_m$, and the weight of each terms in the sum is $n$, by the induction hypothesis we must have $a_i=0, i=1, \cdots m$, which yields the expected contradiction.
\end{proof}

Let $W_{\dm}$ be the subspace of $\QQ\dsmc$ spanned by algebraic subdivision of decorated cones in $\dsmc$. More precisely,
$$
W_{\dm}:=\QQ  \left\{ U-\sum_{i} U_i \,\Big|\, U\in \dsmc \text{ and } \sum_i U_i \text{ is a smooth subdivision of } U \right\}.
$$
\begin{theorem}
The kernel of the linear map
$\Phi_{\dm}: \QQ \dsmc \to \calf_\cm(\QQ)$ is
$W_{\dm}$.
Thus $\Phi_\dm: \QQ \dsmc \to \calf_\cm(\QQ)$
induces a bijective linear map
$$\Phi_\dm: \QQ\overline{\dsmc}:=\QQ\dsmc /W_\dm \cong \calf_\cm(\QQ).$$
\mlabel{thm:wiso}
\end{theorem}

\begin {proof}
Any simple fraction in $\calf_\cm(\QQ)$ is an iterated derivation of a pure fraction in $\cals(\QQ)$ by Proposition~\mref{prop:DerFrac}.(\mref{it:DerFrac}) and, by Proposition~\mref{prop:DerFrac}.(\mref{it:DPhi}), derivations on $\calf_\cm( \QQ)$ are compatible with the derivations on $\QQ \dsmc(\QQ)$ under $\Phi_\dm$.
Then the surjectivity of $\Phi_\dm$ follows from the surjectivity of $\Phi_\cm$ in Theorem~\mref{thm:cfbij}.
Let $\sum_i a_iU_i$ be a smooth subdivision of a decorated closed smooth cone $U$. Then by definition, $U-\sum_{i} a_iU_i$ is an iterated derivation of a $C-\sum_i C_i$ where $\{C_i\}$ is a smooth subdivision of $C$. Thus by Proposition~\mref{prop:DerFrac}~(\mref{it:DPhi}),
$ \Phi(U-\sum_i a_iU_i) = \Phi(U)-\sum_i a_i\Phi(U_i)$ is an iterated derivation of $\Phi(C)-\sum_i \Phi(C_i)$.
By Proposition~\mref{pp:phis}(\mref{it:phis}),
$\Phi(C)-\sum_i \Phi(C_i)=\Phi(C-\sum_i C_i)$ is zero.
Hence its iterated derivations are also zero. Thus $W_\dm\subseteq \ker \Phi_\dm$.

On the other hand, let $F$ be a linear combination of decorated closed cones in $\dsmc$ such that the corresponding linear combination $f:=\Phi(F)\in \calf_\cm (\QQ)$ of fractions is zero. As in the proof of Theorem~\mref{thm:cfbij}, by applying Lemma~\mref{lem:gh}, we can assume that $f$ is homogeneous of degree $k$:
\begin{equation}
f=\sum_{i=1}^r \frac{a_i}{L_{i1}^{s_{i1}}\cdots L_{i\ell}^{s_{i\ell}}},
\mlabel{eq:ffract}
\end{equation}
where $s_{i1}+\cdots+s_{i\ell}=k$, and the linear forms span the same linear spaces.
In other words, $f=\Phi(F)$ where
$$F=\sum_{i=1}^r \frac{a_i}{w(v_{i1},\cdots,v_{i\ell})}
[v_{i1}]^{s_{i1}}\cdots [v_{i\ell}]^{s_{i\ell}}
,$$
where $\{v_{ij}\}_j, 1\leq i\leq r$ span the same linear space in $\QQ^n$.

Choosing smooth subdivisions of the smooth cones $[v_{i1}]\cdots [v_{i\ell}]$ as in Lemma~\mref{lem:ConeSubd} and then taking suitable derivations as in Proposition~\mref{prop:DerCone}.(\mref{it:back}), we obtain a smooth subdivision $\sum_j b_{ij}C_{ij}$ of $[v_{i1}]^{s_{i1}}\cdots [v_{il}]^{s_{il}}$ whose underlying geometric cones meet along faces.
Then
$$ \sum_{i=1}^r \frac{a_i}{w(v_{i1},\cdots,v_{i\ell})} [v_{i1}]^{s_{i1}}\cdots [v_{i\ell}]^{s_{i\ell}} - \sum_{i,j}\frac{a_ib_{ij}}{w(v_{i1},\cdots,v_{i\ell})} C_{ij}$$ lies in $W_\dm$ and we can write
$$\sum_{i,j} \frac{a_ib_{ij}}{w(v_{i1},\cdots,v_{i\ell})} C_{ij}=\sum_{m} b_\ell D_m,$$
where $D_m$ are distinct decorated smooth cones whose underlying geometric cones meet only at faces.
Since $W_\dm$ is in $\ker \Phi_\dm$, we have
\begin{eqnarray*}
0&=&\Phi_\dm\left(\sum_{i=1}^r \frac{a_i}{w(v_{i1},\cdots,v_{i\ell})} [v_{i1}]^{s_{i1}}\cdots [v_{i\ell}]^{s_{i\ell}}\right)\\
&=& \Phi\left(\sum_{ij} \frac{a_ib_{ij}}{w(v_{i1},\cdots,v_{i\ell})} C_{ij} \right)\\
&=&\Phi\left(\sum_{m} b_m D_m\right)\\
&=&\sum_{\ell} b_m \Phi(D_m).
\end{eqnarray*}
By Lemma~\mref{lem:gencone}, all the coefficients $b_m$ are zero. Thus
$$\sum_{i,j}\frac{a_ib_{ij}}{w(v_{i1},\cdots,v_{i\ell})} C_{ij}=\sum_{m} b_m D_m=0.$$
Therefore,
\begin{eqnarray*}
\sum_{i=1}^r \frac{a_i}{w(v_{i1},\cdots,v_{i\ell})} [v_{i1}]^{s_{i1}}\cdots [v_{i\ell}]^{s_{i\ell}}
&=& \sum_{i=1}^r \frac{a_i}{w(v_{i1},\cdots,v_{i\ell})} [v_{i1}]^{s_{i1}}\cdots [v_{i\ell}]^{s_{i\ell}} - \sum_{i,j}\frac{a_ib_{ij}}{w(v_{i1},\cdots,v_{i\ell})} C_{ij}\\
&=&
\sum_{i=1}^r \frac{a_i}{w(v_{i1},\cdots,v_{i\ell})} \left([v_{i1}]^{s_{i1}}\cdots [v_{i\ell}]^{s_{i\ell}} - \sum_{j}b_{ij} C_{ij} \right)
\end{eqnarray*}
lies  in $W_\cm$ and hence $W_{\dm}\supseteq \ker \Phi_{\dm}$.
\end{proof}

\section{Conical zeta values revisited: double subdivision relations and Shintani zeta values}
\mlabel{sec:czv}

The purpose of this section is to revisit conical zeta values using decorated closed cones. We provide a closed cone encoding of CZVs as a generalization of the shuffle encoding of MZVs. We then apply via $\Phi_\cm$ the relation of decorated cones with fractions to give the closed subdivision relation of CZVs and further the double subdivision relation of CZVs by combining with the open subdivision relation in Section 2. We also consider the relationship between CZVs and Shintani zeta values.

\subsection{Closed subdivision relation and shuffle relation}

\subsubsection {Closed cone encoding of CZVs}
We now give another encoding of CZVs by decorated closed cones.

\begin{defn}
Let $[v_1]^{s_1}\cdots[v_k]^{s_k}$ be a decorated smooth closed cone.
\begin{enumerate}
\item
Define the {\bf linearly constrained zeta value (LZV)}
$$\zeta^c([v_1]^{s_1} \cdots [v_k]^{s_k}):=\sum _{m_1=1}^\infty\cdots \sum _{m_r=1}^\infty\frac 1{(a_{11}m_1+\cdots +a_{1r}m_r)^{s_1}\cdots (a_{k1}m_1+\cdots+a_{kr}m_r)^{s_k}}
$$
if the sum is convergent, where $v_i=\sum_{j=1}^r a_{ij}e_j, 1\leq i\leq k$.
\item
Let $\QQ\dsmc_0$ denote the linear subspace of $\QQ\dsmc$ spanned by the set $\dsmc_0$ of decorated smooth cones $[v_1]^{s_1}\cdots [v_k]^{s_k}$ such that $\zeta^c ([v_1]^{s_1}\cdots [v_k]^{s_k})$ is convergent.
\item
Define
$$\zeta^c: \QQ\dsmc_0\to \RR, \quad [v_1]^{s_1}\cdots [v_k]^{s_k}\mapsto \zeta^c([v_1]^{s_1} \cdots [v_k]^{s_k}).$$
\end{enumerate}
\end{defn}

We refer the reader to Section~\mref{ss:shin} for the relationship among CZVs, LZVs and Shintani zeta values.
The following result is immediate from Theorem~\mref{thm:wiso}.

\begin{lemma} The linear map $\zeta^c:\QQ\dsmc_0\to \RR$ is zero on $W_\dm\cap \QQ\dsmc_0$. In particular, let
$$ [v_1]\cdots [v_k] \prec \sum_{i} [v_{i1}]\cdots [v_{ik}]$$
be  a smooth closed subdivision and let
$$D:=[v_1]^{s_1}\cdots [v_k]^{s_k} \prec \sum _i \frac
1{(s_1-1)!\cdots (s_k -1)!}\delta_{v_1^*}^{s_1-1}\cdots
\delta_{v_k^*}^{s_k-1}([v_{i1}]\cdots [v_{ik}])=\sum a_iD_i
$$
be the corresponding algebraic subdivision of  the decorated smooth cone $D$. Then we have
$$\zeta^c(D)=\sum_i a_i \zeta^c(D_i).$$
\mlabel{lem:CSubd}
\end{lemma}
Such a relation is called a {\bf closed subdivision relation}.
\begin{proof}
Note that the coefficients $a_i$ are all positive from the definition of the decorated division. Hence all the terms in
the sum of fractions
$\Phi_\cm(D)=\sum_i a_i \Phi_\cm(D_i)$
are positive. Thus from the convergence of $\zeta^c(D):=\sum_{m_1,\cdots,m_r\geq 1} \Phi_\cm(D)$, we obtain the convergence of $\zeta^c(D_i)$ and the equation in the lemma.
\end{proof}

\subsubsection{Shuffle relations as closed subdivision relations}

Now consider the space $\QQ\dcch$ spanned by the set $\dcch$ of decorated closed Chen cones. Two Chen cones $x_{\sigma(1)}\leq \cdots\leq x_{\sigma(k)}$ and  $x_{\tau(1)}\leq \cdots\leq x_{\tau(l)}$ defined by two permutations $\sigma$ and $\tau$ as in Section~\mref{ss:osub} can only meet along faces where some of the coordinates $x_{\sigma(i)}$ and $x_{\tau(j)}$ coincide.
Thus by Lemma \mref {lem:DecoratedCone}, the set of Chen cones is linearly independent in $\QQ\overline{\dsmc}$.
Thus by Theorem~\mref{thm:wiso}, the linear isomorphism $\overline{\Phi}_\dsmc$ obtained there restricts to a bijection
$$\overline{\Phi}_{ch}: \QQ\dcch \to \calf_{ch},
$$
where
$$ \calf_{ch}:=\QQ\left\{ \frakf\wvec{s_1,\cdots,s_k}{z_1,\cdots,z_k} := \frac{1}{(z_1+\cdots+z_k)^{s_1}\cdots z_k^{s_k}}\,\Big|\, s_i\geq 1,  1\leq i\leq k, k\geq 1\right\}$$
is the space of multiple zeta fractions~\mcite{GX}.
We refer the reader to the introduction for some of the notations used hereafter.

On the other hand, recall~\mcite{H,IKZ} the vector space
$$\calh^{\ssha}_1: = \QQ\,1\oplus \QQ \left \{x_0^{s_1-1}x_1\cdots x_0^{s_k-1}x_1\in \QQ\langle x_0,x_1\rangle\,\Big|\, s_i\geq 1, k\geq 1\right\},
$$
equipped with the shuffle product $\ssha$. See also~\mcite{GX} where it is denoted by $\calh^{\ssha}_1(x_0,x_1)$ and~\mcite{GZ2} where it is denoted by $\calh^0_{\geq 1}$ and is shown to be the free nonunitary Rota-Baxter algebra of weight zero on one generator.
Then the linear map
$$ \frakf: \calh^{\ssha}_1 \to \calf_{ch}(\QQ), \quad x_0^{s_1-1}x_1\cdots x_0^{s_k-1}x_1 \mapsto \frakf\wvec{s_1,\cdots,s_k}{z_1,\cdots,z_k}$$
in~\cite[Theorem 2.1]{GX} together with the obvious bijection
$$\theta: \calh^{\ssha}_1 \to \QQ\dcch,
    x_0^{s_1-1}x_1\cdots x_0^{s_k-1}x_1 \mapsto
    [e_1+\cdots+e_k]^{s_1}\cdots [e_k]^{s_k}$$
gives the commutative diagram
\begin{equation}
 \xymatrix{ \calh^{\ssha} \ar^{\theta}[rrrr] \ar^{\frakf}[rrd] &&&& \QQ\dcch \ar_{\overline{\Phi}_{ch}}[lld] \\
&& \calf_{ch} &&
}
\mlabel{eq:shuf}
\end{equation}
of linear maps. Thus $\frakf$ is a bijection since $\theta$ and $\overline{\Phi}_{ch}$ are.

The shuffle product of the two multiple zeta fractions is induced by $\frakf$~\mcite{GX}. More precisely,
if the shuffle product in $\calh^{\ssha}_1$ is given by
$$ x_0^{s_1-1}x_1\cdots x_0^{s_k-1}x_1 \ssha x_0^{s_{k+1}-1}x_1\cdots x_0^{s_{k+\ell}-1}x_1 =\sum_{w\in
\calh^{\ssha}_1} \alpha_{x_0^{s_1-1}x_1\cdots x_0^{s_k-1}x_1,\, x_0^{s_{k+1}-1}x_1\cdots x_0^{s_{k+\ell}-1}x_1} ^w w,$$
then
$$ \frakf\wvec{s_1,\cdots,s_k}{z_1,\cdots,z_k} \ssha \frakf\wvec{s_{k+1},\cdots,s_{k+\ell}}{z_1,\cdots,z_{k+\ell}} :=\sum_{w\in
\calh^{\ssha}_1} \alpha_{x_0^{s_1-1}x_1\cdots x_0^{s_k-1}x_1,\, x_0^{s_{k+1}-1}x_1\cdots x_0^{s_{k+\ell}-1}x_1} ^w \frakf(w).$$

The shuffle product of two MZVs $\zeta(s_1,\cdots,s_k)$ and $\zeta(s_{k+1},\cdots,s_{k+\ell})$ is determined by the shuffle product of their corresponding multiple zeta fractions
$$ \frakf\wvec{s_1,\cdots,s_k}{z_1,\cdots,z_k}:= \frac{1}{(z_1+\cdots+z_k)^{s_1}\cdots z_k^{s_k}} \text{ and } \frakf\wvec{s_{k+1},\cdots,s_{k+\ell}}{z_1,\cdots,z_{k+\ell}} :=\frac{1}{(z_{k+1}+\cdots+z_{k+\ell})^{s_{k+1}}\cdots z_{k+\ell}^{s_{k+\ell}}}.$$

On the other hand, the decorated cone
$$[e_1+\cdots+e_k]^{s_1}\cdots [e_k]^{s_k}[e_{k+1}+\cdots+e_{k+\ell}]^{s_{k+1}}\cdots [e_{k+\ell}]^{s_{k+\ell}}$$
is  uniquely written as a linear combination of Chen cones
$$\sum_{C\in \dcch}
\alpha_{[e_1+\cdots+e_k]^{s_1}\cdots [e_k]^{s_k},[e_{k+1}+\cdots+e_{k+\ell}]^{s_{k+1}}\cdots [e_{k+\ell}]^{s_{k+\ell}}}^C C.$$
Indeed the inclusion/exclusion principle partitions the domain
$$P_{k,\ell} := \{x_1 \geq \cdots  \geq x_k \geq 0\} \times \{x_{k+1} \geq\cdots \geq x_{k+\ell} \geq 0\} $$
as follows
$$ P_{k,\ell} =\prod_{\sigma \in S_{k,\ell} }
P_\sigma,$$
where $S_{k,\ell}$ is the set of $(k,\ell)$-shuffles and the domain $P_\sigma$  is defined by:
$$P_\sigma = \left\{(x_1,\cdots , x_{k+\ell}) \,\Big| \, x_{\sigma(m)}\geq x_{\sigma(p)}\quad{\rm  if }\quad  m \geq p\quad{\rm  and\ } \sigma(m)\neq \sigma(p) \right\}.$$
So we have
$$[e_1+\cdots+e_k]\cdots [e_k][e_{k+1}+\cdots+e_{k+\ell}]\cdots [e_{k+\ell}] \equiv \sum_{\sigma\in S_{k,\ell}}
[e_{\sigma(1)}+\cdots+e_{\sigma(k+\ell)}]\cdots [e_{\sigma(k+\ell)}] \mod W_\cm.$$
The very definition of subdivisions of decorated cones  then yields
\begin{eqnarray*}
 &&[e_1+\cdots+e_k]^{s_1}\cdots [e_k]^{s_k}[e_{k+1}+\cdots+e_{k+\ell}]^{s_{k+1}}\cdots [e_{k+\ell}]^{s_{k+\ell}} \\
 &\equiv&
\sum_{\sigma\in S_{k,\ell}} \frac{1}{(s_1-1)!\cdots (s_{k+\ell}-1)!} D^{s_1-1}_{(e_1+\cdots+e_k)^*}\cdots D^{s_{k+\ell}-1}_{e_{k+\ell}^*} [e_{\sigma(1)}+\cdots+e_{\sigma(k+\ell)}]\cdots [e_{\sigma(k+\ell)}] \mod W_\dm.
\end{eqnarray*}

By means of the map $\overline{\Phi}_{Ch}$, this gives rise to another way of writing
$\frakf\wvec{s_1,\cdots,s_k}{z_1,\cdots,z_k} \ssha \frakf\wvec{s_{k+1},\cdots,s_{k+\ell}}{z_1,\cdots,z_{k+\ell}}$ as  a sum of other multiple zeta fractions. The multiple zeta fractions are linearly independent as a consequence, for instance, of  the bijectivity of $\overline{\Phi}_{Ch}$.  Thus the two linear combinations must be the same. More precisely,
$$ \alpha_{x_0^{s_1-1}x_1\cdots x_0^{s_k-1}x_1, \, x_0^{s_{k+1}-1}x_1\cdots x_0^{s_{k+\ell}-1}x_1} ^w
=\alpha_{[e_1+\cdots+e_k]^{s_1}\cdots [e_k]^{s_k},[e_{k+1}+\cdots+e_{k+\ell}]^{s_{k+1}}\cdots [e_{k+\ell}]^{s_{k+\ell}}}^{\theta(w)}.$$
Hence we have the following

\begin{theorem} The shuffle relation of multiple zeta fractions and hence of MZVs corresponds via $\overline{\Phi}_{Ch}$ to the subdivision relations of decorated Chen cones. Equivalently, the shuffle product in the shuffle algebra $\calh^{\ssha}_1$ corresponds via $\theta$ in Diagram~(\mref{eq:shuf}) to the subdivision relation of decorated closed Chen cones.
\mlabel{thm:ShuffleAsSubd}
\end{theorem}

\subsection{Double subdivision relations and double shuffle relations}
We now put the open subdivision relation and closed subdivision relation together to form the double subdivision relation for CZVs that generalizes the double shuffle relation for MZVs.

In order to relate the open and closed subdivision relations, we first relate the open and closed decorated cones. We show that they are essentially the transpose of each other.

\begin{defn}
\begin{enumerate}
\item Let $O(\ZZ)$ denote the set of $r\times r$ orthonormal matrices. Let $M\in O(\ZZ)$ and $\vec{s}:=(s_1,\dots,s_r)\in \ZZ^r_{\geq 0}$. Let $v_1,\cdots,v_r$ and $u_1,\cdots,u_r$ be the row and column vectors of $M$. The {\bf (decorated) cone pair} associated with $M$ and $\vec{s}$ is the pair $(C,D)$ consisting of the decorated open cone $C:=C_{M,\vec{s}}=(\ocone{u_1,\cdots,u_r},\vec{s})$ and the decorated closed cone $D:=D_{M,\vec s}=[v_1]^{s_1}\cdots [v_r]^{s_r}$. We call the pair convergent if the corresponding $\zeta$-values $\zeta^0(C)$ and $\zeta^c(D)$  converge.
\item
Let $\dtcp$ denote the set of cone pairs $(C_{M,\vec s},D_{M,\vec{s}})$ where $M\in O(\ZZ)$ and $\vec s\in \ZZ^r_{\geq 0}$. Let
$$p^o: \QQ \dtcp \to \QQ \doc
$$
and
$$p^c:\QQ \dtcp \to \QQ \dsmc
$$
denote the natural projections.
\end{enumerate}
\end{defn}
\begin{remark}
\begin{enumerate}
\item Let $I=\{i_1,\cdots,i_k\}\subseteq [r]$ be the support of $\vec{s}$, namely $I:=\{i\in \{1,\cdots,r\}\,|\, s_i\neq 0\}$. Then $[v_1]^{s_1}\cdots [v_r]^{s_r}=[v_{i_1}]^{s_{i_1}}\cdots [v_{i_k}]^{s_{i_k}}$.
\item
By definition, for a cone pair $(C,D)$, the vectors in $C$ are the column vectors of a matrix whose row vectors are the vectors in $D$. Thus $C$ can be regarded as a transpose of $D$ and will be denoted by $D^\tn$ even though such a $C$ is not unique for a given $D$.
For example, for the decorated closed cone $D:=[e_1]^2[e_1+e_2]=[e_1+e_2][e_1]^2$, we can choose $D^\tn=(\ocone{e_1+e_2,e_2};2,1)$ or $D^\tn=(\ocone{e_1+e_2,e_1};1,2)$.
\item
The central objects of study in this article are  open convex cones   $C=\sum\limits_{i=1}^r \RR_{> 0} v_i$  spanned by vectors $v_i\in \ZZ^k_{\geq 0}, 1\leq i\leq r$. Whereas the order of the vectors  does not play a role for  zeta functions associated with   cones, whether open or closed,  it implicitly does when bringing them together to prove double shuffle relations. We chose to adopt a geometric approach putting geometric  cones   in the forefront when ignoring the order of the generating vectors.  Another possible  and more algebraic focus would be to start off from framed cones, namely cones together with an ordered set of generating vectors, a point of view we intend to explore in a forthcoming paper.
\end{enumerate}
\end{remark}

\begin{lemma}
\begin{enumerate}
\item
The map $p^c$ is surjective. In other words, for any $D\in \dsmc$, there is a decorated open cone $D^\tn$ such that $(D^\tn,D)$ is a cone pair.
\mlabel{it:ocsur}
\item
For any cone pair $(C,D)\in \dtcp$, if $C\in \doc_0$ or $D\in \dsmc_0$, then we have
$$\zeta ^o(C)=\zeta ^c(D).
$$
\mlabel{it:oceq}
\item
Any LZV is a CZV.
\mlabel{it:oclc}
\end{enumerate}
\mlabel{lem:oc}
\end{lemma}

\begin{remark}
Even though a given decorated closed cone $D$ can have multiple decorated open cones $C$ such that $(C,D)$ is a cone pair, by Lemma~\mref{lem:oc}.(\mref{it:oceq}), these decorated open cones give the same CZV.
\end{remark}

\begin{proof}
(\mref{it:ocsur}) Let $D:=[v_1]^{s_1}\cdots [v_k]^{s_k}$ be in $\dsmc$. Then $v_1,\cdots,v_k$ is part of a $\ZZ$ basis $v_1,\cdots,v_k,\cdots,v_r$ of $\ZZ^r_{\geq 0}$. Let $M\in O_{r\times r}(\ZZ)$ be the matrix with $v_1,\cdots,v_r$ as row vectors and let $s_1,\cdots,s_r\in \ZZ_{\geq 0}$ with $s_{k+1}=\cdots =s_r=0$. Then
$C=\left( \langle v_1,\cdots, v_r\rangle, s_1,\cdots, s_r\rangle\right)$ and $D=[v_1]^{s_1}\cdots [v_k]^{s_k}$ build a  cone pair so that $[v_1]^{s_1}\cdots [v_k]^{s_k}$ lies in $\im\, p^c$ .
\smallskip

\noindent
(\mref{it:oceq})
Let $C=(\ocone{u_1,\cdots,u_r};s_1,\cdots, s_r)$ with $u_j=\sum_{i=1}^r a_{ij}e_i, 1\leq i\leq r$, then
$$ n_1e_1+\cdots +n_r e_r= m_1 u_1+\cdots m_r u_r
=  \sum_{j=1}^r \left(\sum_{i=1}^r a_{ij} m_i\right) e_j.$$
Thus $C\cap \ZZ_{\geq 0}^r \subseteq \{(m_1,\cdots,m_k)\,\|\, m_i\geq 1, 1\leq i\leq k\}$. Since the matrix $(a_{ij})$ lies in $O_{r\times r}(\ZZ)$, the inclusion in the other direction also holds.
Thus
\begin{eqnarray*}
\zeta^o(A) &=& \sum_{\vec{n}\in C^\tn\cap\ZZ^k_{\geq 0}} \frac{1}{n_1^{s_1}\cdots n_k^{s_r}} \notag \\
&=& \sum_{m_1,\cdots,m_r\geq 1}  \frac{1}{(\sum_{i=1}^r
c_{i1}m_i)^{s_1}\cdots (\sum_{i=1}^r c_{ir} m_i)^{s_r}}\\
&=&
\zeta^c(B).
\end{eqnarray*}

\noindent
(\mref{it:oclc}) follows from Item~(\mref{it:ocsur}) and Item~(\mref{it:oceq}).
\end{proof}

Let $\dtcp _0$   be the set of {\bf convergent cone pairs} $(C,D)$, i.e., with either $\zeta ^0(A)$ or $\zeta ^c(B)$ convergent. On the grounds of the above lemma, we have restricted  maps:
$$p^o: \QQ \dtcp_0 \to \QQ \doc_0,
$$
and
$$p^c:\QQ \dtcp_0 \to \QQ \dsmc_0
$$
with $p^c$ surjective. Here the subscript $0$ stands for the restriction to the set of cones for which the corresponding conical $\zeta$-values converge.

Thanks to Lemma~\mref{lem:oc}, we obtain a linear map
\begin{equation}
\zeta^c: \QQ\dsmc_0 \to \QQ\ocmzvset,
  \   D \mapsto  \zeta^c(D) =\zeta^o(C),
\mlabel{eq:ccencode}
\end{equation}
where $(C,D)$ is a cone pair.
The following result is immediate from Theorem~\mref{thm:wiso}.

\begin{theorem} Let $(C,D)$ be a convergent cone pair. Let $\{C_i\}_i$ be an open subdivision of the decorated open cone $C$ and let $\sum_j c_j D_j$ be a subdivision of the decorated closed cone $D$. Also let $D_j^\tn\in \doc$ be the transpose cone of $D^j$, that is, $(D_j^\tn,D_j)$ is a cone pair. Then
\begin{equation}
\sum_i C_i - \sum_j c_j D_j^\tn
\mlabel{eq:dbsubd}
\end{equation}
lies  in the kernel of $\zeta^o$.
\mlabel{thm:CSubd}
\end{theorem}
A linear combination in Eq.~(\mref{eq:dbsubd}) is called a {\bf double subdivision relation} and  we shall see that it generalizes the usual double shuffle relation.

\begin{proof}
By Lemma~\mref{lem:subd}, for the open subdivision $\{C_i\}_i$ of the decorated open cone $C$, we have
\begin{equation}
\zeta^o(C)=\sum_{i} \zeta^o(C_i).
\mlabel{eq:sumo}
\end{equation}
On the other hand, for the closed subdivision $\sum_{j}a_j D_j$ of the decorated closed cone $D$, by Lemma~\mref{lem:CSubd}, we have
\begin{equation}
\zeta^c(D)=\sum_j a_j \zeta^c(D_j).
\mlabel{eq:sumc}
\end{equation}
Then by Lemma~\mref{lem:oc}.(\mref{it:ocsur}), there are open decorated cones $D_j^\tn$ such that $(D_j^\tn,D_j)$ is in $\dtcp_0$. Further by Lemma~\mref{lem:oc}.(\mref{it:oceq}), we also have
$\zeta^c(D)=\zeta^o(C)$ and $\zeta^c(D_j)=\zeta^o(D_j^\tn)$. Therefore combining with Eqs.~(\mref{eq:sumo}) and (\mref{eq:sumc}) we obtain
$$\sum_i\zeta^o(C_i)=\zeta^o(C)=\sum_ja_j \zeta^o(D_j^\tn),$$
as needed.
\end{proof}

In summary, we have the following commutative diagram where the subscript $0$ stands for restricting to cones of which the corresponding zeta values are convergent. The inner triangle gives the double shuffle relation. There $\calh^\ast_0$ is the quasi-shuffle algebra defined in Eq.~(\mref{eq:qsh}).
The left part of the diagram is from Eq.~(\mref{eq:osub}).
The commutativity of the outer triangle follows from Lemma~\mref{lem:oc}. The map $\zeta^{\ssha}$ is the usual shuffle encoding of MZVs, here expressed as the composition of $\theta$ with the free summation : $$\zeta^{\ssha}(x_0^{s_1-1}x_1\cdots x_0^{s_k-1}x_1)=
\sum_{m_1,\cdots,m_k\geq 1} \frac{1}{(m_1+\cdots+m_k)^{s_1}\cdots m_k^{s_k}}=\zeta(s_1,\cdots,s_k).$$
The map $\tn$ is well-defined on $\QQ\dcch$ from the standard form of Chen cones:
$$\tn ( [e_{\sigma(1)}+\cdots+e_{\sigma(n)}]^{s_1}\cdots [e_{\sigma(n)}]^{s_n}) =(\ocone{e_{\sigma (1)},\cdots, e_{\sigma (1)}+\cdots+e_{\sigma(n)}};s_{1},\cdots, s_{n}).$$

\nc{\domc}{\mathcal{DMC}^o}  
\nc{\dcmc}{\mathcal{DMC}^c}  
\nc{\docc}{\mathcal{DCh}^o} 
\nc{\dccc}{\mathcal{DCh}^c} 

\begin{equation}
\xymatrix{
\QQ\doc_0 \ar@/_3pc/[rrrddddd]^{\zeta^o} &&& \QQ\dtcp_0 \ar[lll]_{p^o} \ar[rrr]^{p^c} &&& \ \QQ\dcmc_0  \ar@/^3pc/[lllddddd]_{\zeta^c}\\
&& \QQ\docc_0 \ar@{_{(}->}[ull] \ar@{>->>}[d] && \ \QQ\dccc_0 \ar@{>->>}[ll]_{\tn} \ar@{^{(}->}[urr] \ar@{>->>}[d]&& \\
&& \calh^\ast_0 \ar[ddr]^{\zeta^*} && \ \calh^{\ssha}_0 \ar@{>->>}[ll]_{\eta} \ar[ddl]_{\zeta^{\ssha}} &&\\
&&&&&&\\
&&& \QQ\,\mzvset \ar@{_{(}->}[d] &&&\\
&&& \QQ\,\ocmzv &&&
}
\mlabel{eq:dsub}
\end{equation}
In Corollary~\mref{co:surj} we will prove the surjectivity of $\zeta^o$ and $\zeta^c$, as in the case of MZVs.

\begin{defn}
{\rm
For any not necessarily convergent cone pair $(C,D)$, let $\{C_i\}$ be a subdivision of $C$ and $\sum_j a_jD_j$ a subdivision of $D$. If $\sum_{i}C_i-\sum_j a_j D_j^\tn$ is in $\QQ\domc_0$, then it is called an {\bf extended double subdivision relation}.
}
\mlabel{de:dsubd}
\end{defn}

\begin{conjecture}
The kernel of $\zeta^o$ is the subspace $I_{EDS}$ of $\QQ\domc$ generated by the extended double subdivision relations.
\end{conjecture}

In view of the Double Shuffle Conjecture for MZVs, we also make the following
\begin{conjecture}
The intersection $I_{EDS}\cap \QQ\doch$, identified with a subset of $\calh^\ast_0$ by the bijection $\QQ\doch\to \calh^\ast_0$, is the extended double shuffle ideal $I_{EDS}$ of $\calh^\ast_0$.
\end{conjecture}
Verifying this conjecture would provide evidence for the Double Shuffle Conjecture.

\subsubsection{Examples}
We demonstrate the utility of double subdivision relations by some examples. The first example illustrates how the double shuffle relation of MZVs can be obtained when subdivisions of open and closed cones are used in place of quasi-shuffles and shuffles.

\begin {exam}
Consider the cone pair $(C,D)$ associated with $M=I_{2\times 2}\in O_{2\times 2}(\ZZ)$ and $\vec s=(2,2)$. Then $C=(\ocone{e_1,e_2};2,2)$ and $D=[e_1]^2[e_2]^2$.
From the open subdivision
$$\ocone{e_1,e_2}=\ocone{e_1, e_1+e_2}\sqcup\ocone{e_2, e_1+e_2}\sqcup \ocone{e_1+e_2},$$
we obtain
\begin{eqnarray*}
\zeta(2)\zeta(2)&=&\zeta^o(\ocone{e_1,e_2};2,2)\\ &=&\zeta^o(\ocone{e_1, e_1+e_2};2,2)+\zeta^o(\ocone{e_2, e_1+e_2};2,2)+\zeta^o(\ocone{e_1+e_2};2,2)\\
&=&2\zeta (2,2)+\zeta(4),
\end{eqnarray*}
giving the quasi-shuffle relation for $\zeta (2)\zeta (2)$.

On the other hand, we have
$$\zeta^c([e_1]^2[e_2]^2)=\zeta (2)\zeta(2).
$$
Differentiating the subdivision
$$ [e_1][e_2]\prec [e_1][e_1+e_2]+[e_2][e_1+e_2]$$
we have the decorated subdivision
$$[e_1]^2[e_2]^2 \prec [e_1]^2[e_1+e_2]^2+2[e_1][e_1+e_2]^3+2[e_2][e_1+e_2]^3+[e_2]^2[e_1+e_2]^2.
$$
Applying $\zeta^c$ we obtain the shuffle relation
$$\zeta (2)\zeta (2)=4\zeta (3,1)+2\zeta (2,2).
$$
Altogether, we recover the double shuffle relation
$\zeta(4)=4 \zeta(3,1).$
\end{exam}

The next example provides an alternative way to apply the double subdivision relation to get a double shuffle relation of MZVs, that bypasses the stuffle and shuffle products in that there is no presence of products of MZVs. Compare with the previous example.

\begin {exam}{\rm Consider the cone pair $((\ocone{e_1,e_1+e_2};2,2), [e_1+e_2]^2[e_2]^2)$ associated with  the matrix $\left(\begin{array}{cc} 1& 1\\ 0&1\end{array}\right)$ and $\vec s= (2,2)$.
The open cone subdivision
$$\ocone{e_1,e_1+e_2} = \ocone{e_1,2e_1+e_2}\sqcup \ocone{2e_1+e_2,e_1+e_2}\sqcup \ocone{2e_1+e_2}$$
gives the following subdivision ofthe  decorated open cone
\begin{equation}
(\ocone{e_1,e_1+e_2};2,2) \prec(\ocone{e_1,2e_1+e_2};2,2) +(\ocone{2e_1+e_2,e_1+e_2};2,2) +(\ocone{2e_1+e_2};2,2).\mlabel{eq:eg3}
\end{equation}
On the other hand, the closed cone subdivision $[e_1+e_2][e_2]\prec [e_1+e_2][e_1+2e_2]+[e_1+2e_2][e_2]$ induces the following subdivision of the decorated closed cone
$$[e_1+e_2]^2[e_2]^2\prec [e_1+e_2]^2[e_1+2e_2]^2+2[e_1+e_2][e_1+2e_2]^3+[e_2]^2[e_1+2e_2]^2+2[e_2][e_1+2e_2]^3.
$$
Taking the transposes of the right hand side we have
$$(\ocone{e_1+e_2, 2e_1+e_2};2,2)
+2(\ocone{e_1+e_2, 2e_1+e_2};3,1)
+(\ocone{e_1,2e_1+e_2};2,2)+2(\ocone{e_1,2e_1+e_2};3,1).
$$
Combining with Eq.~(\mref{eq:eg3}) we obtain the double subdivision relation
$$(\ocone{2e_1+e_2};2,2)-2(\ocone{e_1+e_2, 2e_1+e_2};3,1) -2(\ocone{e_1,2e_1+e_2};3,1).$$
Then by Theorem~\mref{thm:CSubd}), we have
$$\zeta^o(\ocone{2e_1+e_2};2,2)=2\zeta^o(\ocone{e_1+e_2, 2e_1+e_2};3,1)+2\zeta^o(\ocone{e_1,2e_1+e_2};3,1).$$
Since $\zeta^o(\ocone{2e_1+e_2},(2,2))=\frac{1}{4}\zeta(4)$ and
\begin{eqnarray*}
&&\zeta^o(\ocone{e_1+e_2, 2e_1+e_2};3,1)+\zeta^o(\ocone{e_1,2e_1+e_2};3,1)\\
&=&
\zeta^o(\ocone{e_1,e_1+e_2};3,1) -\zeta^o(\ocone{2e_1+e_2};3,1)\\
&=& \zeta(3,1)-\zeta(4),
\end{eqnarray*}
we derive the formula
$\zeta (4)=4\zeta (3,1)$
in a way that differs  from the double shuffle approach.
}
\end{exam}

\medskip
The next example shows that double subdivision relations can give relations between MZVs and CZVs.

\begin {exam} {\rm
We give a double subdivision relation of the cone pair $$((\ocone{e_1,e_1+e_2};2,1),[e_1+e_2]^2[e_2]).$$
An subdivision of the  open cone $\ocone{e_1,e_1+e_2}$ yields
\begin{equation}
(\ocone{e_1,e_1+e_2};2,1)
\prec (\ocone{e_1, 2e_1+e_2};2,1)
+(\ocone{2e_1+e_2, e_1+e_2};2,1) +(\ocone{2e_1+e_2};2,1).
\mlabel{eq:eg2}
\end{equation}

On the other hand, we deduce from the subdivision $$[e_1+e_2][e_2]\prec[e_1+2e_2][e_2]+[e_1+e_2][e_1+2e_2]$$
the following subdivision for the decorated closed cone $[e_1+e_2]^2[e_2]$
$$[e_1+e_2]^2[e_2]\prec [e_1+2e_2]^2[e_2]+[e_1+e_2]^2[e_1+2e_2]+[e_1+2e_2]^2[e_1+e_2].
$$
Taking the transposes we get
$$(\ocone{e_1, 2e_1+e_2};2,1) +(\ocone{e_1+e_2, e_1+2e_2};2,1)+(\ocone{e_1+e_2, 2e_1+e_2};2,1).
$$
Combining with Eq.~(\mref{eq:eg2}) we obtain the double subdivision relation
$$
(\ocone{2e_1+e_2};2,1)-(\ocone{e_1+e_2, e_1+2e_2};2,1).
$$
By Theorem~\mref{thm:CSubd}, we have
$$\zeta^o(\ocone{2e_1+e_2};2,1)=\zeta^o(\ocone{e_1+e_2, e_1+2e_2};2,1).$$
Noting that
$\zeta^o(\ocone{2e_1+e_2},(2,1))=\frac{1}{4}\zeta(3)$. Then we obtain the following analog of the Euler sum formula.
$$\zeta (3)=4\zeta^o(\ocone{e_1+e_2, e_1+2e_2};2,1).
$$
}
\end{exam}

\subsection{Conical zeta values and Shintani zeta values}
\mlabel{ss:shin}
We show that CZVs span the same space as the space of Shintani zeta values.

\subsubsection{Fractions and smooth fractions}

\begin {prop}
\begin{enumerate}
\item
Any fraction of the form $\frac{1}{L_1^{s_1}\cdots L_k^{s_k}}$ where $L_1,\cdots,L_k$ are linear forms in $\deff [\vec{z}]$ is a linear combination of pure fractions.
\mlabel{it:genfraca}
\item
Any fraction of the form $\frac{1}{L_1^{s_1}\cdots L_k^{s_k}}$ where $L_1,\cdots,L_k$ are linear forms in $\deff [\vec{z}]$ with non-negative coefficients is a positive linear combination of pure fractions whose linear forms have positive non-negative coefficients.
\mlabel{it:genfracb}
\end{enumerate}
\mlabel {pp:GenFrac}
\end{prop}

\begin {proof}
(\mref{it:genfraca})
By the unique factorization in $\deff[\vec{z}]$, we can uniquely write such a fraction as $f=\frac{a}{L_1^{s_1}\cdots L_k^{s_k}}$, where the $L_i$'s are not multiples of one another. We then denote $d:=d(f)=k-\rk\{L_1,\cdots,L_k\}.$

We prove Item~(\mref{it:genfraca}) by induction on $d$. If $d=0$, then $L_1,\cdots,L_k$ are already linearly independent and we are done.

Suppose Item~(\mref{it:genfraca}) has been proved for all fractions with $d= m\geq 0$ and consider a fraction $f$ with $d=m+1$. Then $d>0$ and $k>\rk\{L_1,\cdots,L_k\}$. Thus there is a linearly independent subset $L_{i_1}, \cdots , L_{i_\ell}$ of $\{L_1, \cdots , L_k\}$ and $L_{i_{\ell+1}}\not\in\{L_{i_1},\cdots,L_{i_\ell}\}$ such that
\begin{equation}
L_{i_{\ell+1}}=\sum_{j=1}^\ell a_jL_{i_j},\quad a_j\neq 0, 1\leq j\leq \ell.
\mlabel{eq:lindb}
\end{equation}
Since the linear forms $L_i$ are not multiples of one another, we have $\ell\geq 2$. Thus
\begin{equation}
\frac 1{L_{i_1}\cdots L_{i_\ell}}=\frac {L_{i_{\ell+1}}}{L_{i_1}\cdots L_{i_\ell}L_{i_{\ell+1}}}= \sum_{j=1}^\ell \frac {a_i}{L_{i_1}\cdots \hat{L_{i_j}}\cdots L_{i_\ell}L_{i_{\ell+1}}},
\mlabel{eq:lind2}
\end{equation}
where $\hat{L_{i_j}}$ indicates that the factor $L_{i_j}$ is absent. Implementing this procedure reduces the total degree of $L_{i_1}, \cdots, L_{i_\ell}$ by one. An  iteration of  this substitution procedure   inside $f$,  terminates after finitely many times, when exactly one of the $L_{i_j}$'s ($1\le j\le \ell$) disappears.

Each term in the resulting sum is of the form
$g:=\frac{1}{L_{i_1}^{t_1}\cdots \hat{L_{i_j}}^{t_j} \cdots L_{i_k}^{t_k}}, 1\leq j\leq \ell$. Hence the linear forms are still not multiples of one another.
Furthermore, by Eq.~(\mref{eq:lindb}), we have
$$\rk \{L_{i_1},\cdots,L_{i_\ell}\} =\rk\{L_{i_1},\cdots,\hat{L}_{i_j},\cdots,L_{i_\ell}, L_{i_{\ell+1}}\} =\ell, \quad 1\leq j\leq \ell.$$
Hence
$$\rk\{L_1,\cdots,L_k\} =\rk\left(\{L_1,\cdots,\cdots,L_k\}\backslash \{L_{i_j}\}\right), \quad 1\le j\le \ell.$$
Thus
$$ d(g)=k-1-\rk\left(\{L_1,\cdots,\cdots,L_k\}\backslash \{L_{i_j}\}\right) =k-1-\rk\{L_1,\cdots,L_k\}=d(f)-1=m.$$
It follows from the induction hypothesis,that each term $g$ in the resulting sum is a linear combination of pure fractions. Then the same is true for $f$ itself. This completes the induction.
\smallskip

(\mref{it:genfracb}) For $f$ in Item~(\mref{it:genfracb}) define $d(f)$ in the same way. We prove the statement by induction on $d$. If $d=0$, then $L_1,\cdots,L_k$ are already linearly independent and we are done.

Suppose the proposition has been proved for all fractions with $d= m\geq 0$ and consider a fraction $f=\frac{1}{L_1^{s_1}\cdots L_k^{s_k}}$ with $d=m+1$. Define
$$\delta(f):= \min\left\{p\in \ZZ_{\geq 0}\,\Big|\, \begin{array}{l}
\exists\, \text{linearly independent } L_{i_1}, \cdots L_{i_\ell} \text{ and } L_j\not\in \{L_{i_1},\cdots,L_{i_\ell}\} \text{ such that} \\
L_j=\sum_{j=1}^\ell c_j L_{i_j}, \text{ with } c_j\neq 0, 1\leq j\leq \ell \text{ and } p \text{ of which being negative}
\end{array}  \right \}.$$
We will prove Item~(\mref{it:genfracb}) for $f$ with $d=m+1$ by induction on $\delta(f)\geq 0$. Since the linear forms $L_i$ are not multiples of one another, we have $\ell\geq 2$. If $\delta(f)=0$, then there is a linearly independent set $\{L_{i_1},\cdots,L_{i_\ell}\}$ and $L_{i_{\ell+1}}$ not in this set such that $L_{i_{\ell+1}}=\sum_{j=1}^\ell a_j L_{i_j}$ with positive coefficients. Then following the substitutions in Item~(\mref{it:genfraca}), we find that $f$ is linear combination of fractions of the form $g:=\frac{1}{L_{i_1}^{t_1}\cdots \hat{L_{i_j}}^{t_j} \cdots L_{i_k}^{t_k}}, 1\leq j\leq \ell$. Since all the $a_j$'s are positive, the coefficients of this linear combination is also positive. Further as in the proof of Item~(\mref{it:genfraca}), by the induction hypothesis, each term $g$ is a linear combination, with positive coefficients, of pure fractions whose linear forms have positive coefficients. Thus $f$ is also such a linear combination. This completes the induction on $d$ in the case when $\delta(f)=0$.

Suppose the induction on $d$ is proved when $\delta(f)=q\geq 0$ and consider a fraction $f=\frac{1}{L_1^{s_1}\cdots L_k^{s_k}}$ with $\delta(f) = q+1$. Let $\{L_{i_1},\cdots,L_{i_\ell}\}$ be independent and $L_{i_{\ell+1}}\not\in \{L_{i_1},\cdots,L_{i_\ell}\}$ such that Eq.~(\mref{eq:lindb}) holds with $q+1$ negative coefficients. Reordering the linear forms if necessary, we can assume $a_1<0$. Applying
$$ \frac{1}{L_{i_1}L_{i_{\ell+1}}} =\frac{1}{L_{i_i}(-a_1L_{i_1}+L_{i_{\ell+1}})} + \frac{-a_1}{(-a_1L_{i_1}+L_{i_{\ell+1}})},$$
we obtain
$$ \frac{1}{L_{i_1}^{s_1}L_{i_{\ell+1}}^{s_{\ell+1}}} =\sum_j \frac{\alpha_j} {L_{i_1}^{u_j}(-a_1L_{i_1}+L_{i_{\ell+1}})^{v_j}} +\sum_j\frac{\beta_j}{(-a_1L_{i_1}+L_{i_{\ell+1}})^{w_j} L_{i_{\ell+1}}^{x_j}}$$
with non-negative $\alpha _i$ and $\beta _j$'s.
Hence
\begin{eqnarray}
&&\frac{1}{L_{1}^{s_1}\cdots L_{k}^{s_{k}}} \mlabel{eq:pcomb}\\
&=&\sum_j \frac{\alpha_j} {L_{i_1}^{u_j}(-a_1L_{i_1}+L_{i_{\ell+1}})^{v_j} L_2^{s_2}\cdots \widehat{L_{i_{\ell+1}}^{s_{i_{\ell+1}}}}\cdots L_k^{s_k}} +\sum_j\frac{\beta_j}{(-a_1L_{i_1}+L_{i_{\ell+1}})^{w_j} L_{i_{\ell+1}}^{x_j}L_2^{s_2}\cdots \widehat{L_{i_{\ell+1}}^{s_{i_{\ell+1}}}}\cdots L_k^{s_k}}.
\notag
\end{eqnarray}

If $-a_1L_{i_1}+L_{i_{\ell+1}}$ is a multiple of an $L_j$ (this can happen for example when $\ell=2$), then since both $-a_1L_{i_1}+L_{i_{\ell+1}}$ and $L_j$ are linear forms with positive coefficients, this multiple must also be positive. Thus in the above two sums, we can replace the power of $L_j$ by that of $-a_1L_{i_1}+L_{i_{\ell+1}}$ and still get a linear combination of positive coefficients.
Thus we can assume that the linear forms in each term $g$ of the two sums in Eq.~(\mref{eq:pcomb}) are not multiples of one another. Then the relation $-a_1L_{i_1}=\sum_{j=2}^\ell a_j L_{i_j}$ in $g$ shows that $\delta(g)\leq q$ and hence the induction hypothesis on $\delta(g)$ applies and expresses $g$ as a positive linear combination prescribed in Item~(\mref{it:genfracb}). Since the coefficients in Eq.~(\mref{eq:pcomb}) are positive, $f$ also has the desired positive linear combination. This completes the induction for $\delta(f)$ and hence for $d$.
\end{proof}

\subsubsection{Conical zeta values and Shintani zeta values}

Recall that for a matrix $M=(c_{ij})\in M_{k\times r}(\ZZ _{\ge 0})$, a {\bf Shintani zeta value (SZV)} is the special value of the {\bf Shintani
zeta function}~\mcite{Ma}

\begin{equation}
\zeta(M;s_1,\cdots,s_k)= \sum_{m_1=1}^\infty \cdots \sum_{m_r=1}^\infty
\frac{1}{(c_{11}m_1+\cdots+c_{1r}m_r)^{s_1}\cdots (c_{k1}m_1+\cdots
+ c_{kr}m_r)^{s_k}},
\label{eq:smzv}
\end{equation}
at $(s_1,\cdots,s_k)\in \ZZ^{k}$ if the sum converges.

\begin{prop}
\begin{enumerate}
\item
Let $C=\langle v_1,\cdots,v_r\rangle $ be a smooth cone and let $M:=M_C$ be the matrix with the spanning set $\{v_1,\cdots,v_r\}$ of $C$ as  column vectors. Then
$$\zeta^o(C;s_1,\cdots,s_k)=\zeta (M;s_1,\cdots,s_k).$$
\mlabel{it:smzv1}
\item
Any CZV is a linear combination of Shintani zeta values.
\mlabel{it:smzv2}
\end{enumerate}
\mlabel{pp:smzv}
\end{prop}

\begin{proof}
(\mref{it:smzv1})
Since $C=\ocone{v_1,\cdots,v_r}$ is a
smooth cone generated by primary set $\{v_1,\cdots,v_r\}$,
the set can be expanded to a $\ZZ$-basis $\{v_1,\cdots,v_k\}$ of $\ZZ^k$. Write $v_i=\sum_{j=1}^k c_{ij}e_j, 1\leq i\leq k$. Then $(c_{ij})\in M_{k\times k}(\ZZ_{\geq 0})$ has determinant $1$.
Then from
$\sum_{i=1}^k n_ie_i=\sum_{j=1}^r m_i v_i =\sum_{j=1}^k m_i v_i$ where $m_i=0$ for $r+1\leq i\leq k$, we have
$$(n_1,\cdots,n_k)^\tn=(c_{ji}) (m_1,\cdots,m_k)^\tn,$$
and hence
$$(m_1,\cdots,m_k)^\tn=(c_{ji})^{-1} (n_1,\cdots,n_k)^\tn.$$
Thus we have
\begin{eqnarray*}
\zeta^o(C;s_1,\cdots,s_k) &=& \sum_{\vec{n}\in C} \frac{1}{n_1^{s_1}\cdots n_k^{s_k}} \notag \\
&=& \sum_{m_1,\cdots,m_k\geq 1}  \frac{1}{(\sum_{i=1}^r
c_{i1}m_i)^{s_1}\cdots (\sum_{i=1}^r c_{ik} m_i)^{s_k}}\\
&=&
\zeta(M;s_1,\cdots,s_k).
\end{eqnarray*}
\smallskip

\noindent
(\mref{it:smzv2})
Since any rational cone $C$ can be subdivided into smooth cones
$\{C_1,\cdots,C_p\}$ by Proposition~\mref{pp:ratsm}, we have
$$ \zeta^o(C;s_1,\cdots,s_k)= \sum_{i=1}^p \zeta^o({C_i};s_1,\cdots,s_k),$$
hence is a linear combination of SZVs.
\end{proof}

\begin {defn} {\rm
A matrix $M=(c_{ij})\in M_{k\times n}(\ZZ _{\ge 0})$ is of maximal rank if the rank of the  matrix is $k$. The corresponding SZV will be called a {\bf maximal rank SZV}.
}
\end{defn}

\begin {lem}
\begin{enumerate}
\item
A convergent SZV is a $\QQ$-linear combination of convergent maximal rank SZVs.
\mlabel {it:MaxSZV1}
\item
Any convergent maximal rank SZV is a $\QQ$-linear combination of convergent LZVs.
\mlabel{it:MaxSZV2}
\end{enumerate}
\mlabel {lem:MaxSZV}
\end{lem}

\begin{proof}
(\mref{it:MaxSZV1})
For $M=(c_{ij})\in M_{k\times n}(\ZZ _{\ge 0})$,
$$
\zeta(M;s_1,\cdots,s_k)=\hspace{-.5cm} \sum_{m_1,\cdots,m_n>0}
\frac{1}{(c_{11}m_1+\cdots+c_{1n}m_n)^{s_1}\cdots (c_{k1}m_1+\cdots
+ c_{kn}m_n)^{s_k}}.
$$
By Proposition~\mref{pp:GenFrac}.(\mref{it:genfracb}), the fraction $\frac{1}{(c_{11}m_1+\cdots+c_{1n}m_n)^{s_1}\cdots (c_{k1}m_1+\cdots
+ c_{kn}m_n)^{s_k}}$ is a positive linear combination of fractions in which the linear forms are linearly independent and have positive coefficients.
Since $\zeta(M;s_1,\cdots,s_k)$ is convergent, all these fractions give rise to convergent SZVs,  leading to the desired linear combination.
\smallskip

\noindent
(\mref{it:MaxSZV2}) For $M=(c_{ij})\in M_{k\times n}(\ZZ _{\ge 0})$ with maximal rank, i.e. $rank(M)=k$, we know that the cone
spanned by vectors $v_1=(c_{11},\cdots c_{1n}), \cdots , v_k=(c_{k1},\cdots c_{kn})$ can be subdivided into smooth cones $C_i=[v^{(i)}_1]\cdots [v^{(i)}_k]$, $v^{(i)}_j=(c^{(i)}_{j1},\cdots ,c^{(i)}_{jn})$, $i=1,\cdots, l$, so we have   subdivision relations,
$$\frac{1}{(c_{11}m_1+\cdots+c_{1n}m_n)\cdots (c_{k1}m_1+\cdots
+ c_{kn}m_n)}=\sum _i \frac {a_i}{(c^{(i)}_{11}m_1+\cdots+c^{(i)}_{1n}m_n)\cdots (c^{(i)}_{k1}m_1+\cdots
+ c_{kn}m_n)}
$$
where the $a_i$'s are positive rational numbers.

Taking derivatives and applying Proposition \mref {prop:DerFrac} yields the conclusion.
\end{proof}

\begin {theorem} Convergent CZVs, LZVs and SZVs span the same linear space over $\QQ$. Thus
$$\QQ \czvset =\QQ \lzvset =\QQ \szvset.$$
\end{theorem}
\begin{proof}
By Lemma~\mref{lem:MaxSZV}, we have
$\QQ\szvset \subseteq \QQ \lzvset$. By Lemma~\mref{lem:oc} we have $\QQ\lzvset \subseteq \QQ\czvset.$ By Proposition~\mref{pp:smzv}.(\mref{it:smzv2}), we have $\QQ\czvset \subseteq \QQ\szvset.$
\end{proof}

\begin{coro}
The linear maps $\zeta^o$ and $\zeta^c$ are surjective.
\mlabel{co:surj}
\end{coro}

\noindent
{\bf Acknowledgements}:
L.~Guo acknowledges support from NSF grant DMS 1001855. B. Zhang thanks support from NSFC grant 11071176 and 11221101.

\bibliographystyle{plain}

\end{document}